%% file: bergbau.tex
\documentclass[12pt,a4paper,bibliography=totoc]{scrartcl}
\usepackage{scrpage2} 
\usepackage{amsmath, amssymb, latexsym}
\usepackage{algorithm}
\usepackage{algpseudocode}
\usepackage{xcolor}

\usepackage{float}
\usepackage{graphicx}

\usepackage{amsthm}

\newtheorem{theorem}{Theorem}

\newtheorem{lemma}[theorem]{Lemma}
\newtheorem{prop}[theorem]{Proposition}

\newtheorem{remark}[theorem]{Remark}

\numberwithin{theorem}{section}
\numberwithin{equation}{section}

\usepackage{prettyref} 
\newrefformat{thm}{Theorem~\ref{#1}}
\newrefformat{cor}{Corollary~\ref{#1}}
\newrefformat{ex}{Example~\ref{#1}}
\newrefformat{prop}{Proposition~\ref{#1}}
\newrefformat{prob}{Problem~\ref{#1}}
\newrefformat{part}{Part~\ref{#1}}
\newrefformat{sect}{Section~\ref{#1}}
\newrefformat{def}{Definition~\ref{#1}}
\newrefformat{rem}{Remark~\ref{#1}}
\newrefformat{ass}{Assumption~\ref{#1}}
\newrefformat{tab}{Table~\ref{#1}}

\usepackage{varioref} 

\usepackage{siunitx} 
\usepackage{enumitem} 
\usepackage{booktabs} 
\usepackage{mathtools} 

\usepackage[unicode=true,
 bookmarks=true,bookmarksnumbered=false,bookmarksopen=false,
 breaklinks=true,pdfborder={0 0 0},backref=page]
 {hyperref}
\hypersetup{pdftitle={Heuristic Approximations for Closed Networks: A Case Study in Open-pit Mining},
 pdfauthor={Hans Daduna, Ruslan Krenzler, Dietrich Stoyan, Robert Ritter},
 pdfsubject={queueing systems},
 pdfkeywords={queueing systems, mining, mathematical modeling},
 colorlinks=true, 
 linkcolor=blue,
 urlcolor=blue,
 citecolor=blue,
 draft=false}
 
\usepackage[capitalise]{cleveref} 
 
\newcommand{\beq}{\begin{equation}}
\newcommand{\eeq}{\end{equation}}
\newcommand{\beqo}{\begin{equation*}}
\newcommand{\eeqo}{\end{equation*}}
\newcommand{\bdm}{\begin{displaymath}}
\newcommand{\edm}{\end{displaymath}}
\newcommand{\beqar}{\begin{eqnarray}}
\newcommand{\eeqar}{\end{eqnarray}}
\newcommand{\beqaro}{\begin{eqnarray*}}
\newcommand{\eeqaro}{\end{eqnarray*}}
\newcommand{\bal}{\begin{align}}
\newcommand{\eal}{\end{align}}

\newcommand{\X}{\mathbf{X}}

\newcommand{\mmu}{\mu^{-1}}
\newcommand{\Jset}{\overline{J}}
\newcommand{\nvect}{\mathbf{n}}

\DeclareMathOperator*{\argmin}{\arg\!\min}
\ihead{Daduna, Krenzler, Ritter, Stoyan}
\ohead{25 March 2016}
\setheadsepline{0.4pt}

\newenvironment{ignorechecktex}{}{}
\newcommand\alignphantomspace{\mathrel{\hphantom{=}}}
\newcommand{\expdist}{Exp} 
\begin{document}
\begin{ignorechecktex}
\title{Heuristic Approximations for Closed Networks:\\ A Case Study in Open-pit Mining
}

\date{25 March 2016}
\author{Hans Daduna%
\thanks{\emph{Hamburg University},
    Department of Mathematics,
    Bundesstra\ss e 55,
    20146 Hamburg,
    Germany}
\and
Ruslan Krenzler%
\footnotemark[1] 
\and
Robert Ritter%
\thanks{\emph{Sandvik Mining and Construction Crushing Technology GmbH},
	Bautzner Stra\ss e 58,
	01099 Dresden,
    Germany}
\and Dietrich Stoyan%
\thanks{\emph{TU Bergakademie Freiberg}, Institute of Stochastics,
    Pr{\"u}ferstra\ss e 6,
    09596 Freiberg,
    Germany}
}

\maketitle
\end{ignorechecktex}
\begin{abstract}
We investigate a fundamental model from open-pit mining, which is a cyclic system consisting of a shovel,
traveling loaded, unloading facility, and traveling back empty. The interaction of these subsystem determines the capacity of the shovel, which is the fundamental quantity of interest. To determine this capacity one needs the stationary probability that the shovel is idle. Because an exact analysis of the performance of the system is out of reach, besides of simulations there are various approximation algorithms proposed in the literature which stem from
computer science and can be characterized as general purpose algorithms.
We propose for solving the special problem under mining conditions an extremely simple algorithm.
Comparison with several general purpose algorithms shows that for realistic situations the special algorithm outperforms the precision of the general purpose algorithms.
This holds even if these general purpose candidates incorporate more details of the underlying models
than our simple algorithm, which works on a strongly reduced model.
The comparison and assessment is done with extensive simulations on a level of detail which the general purpose algorithms are able to cover.
 
\bigskip
\noindent
\textbf{AMS (1991) subject classification:}  60K25, 60J25

 
\end{abstract}

\tableofcontents

\section{Introduction}\label{sect:intro}
Closed networks of queues served in many areas as models to investigate performance and reliability of
systems.
For so-called product form networks there exist well developed tool sets for such investigations,
see the classical papers on  Jackson~\cite{jackson:57} and Gordon and Newell~\cite{gordon;newell:67} networks
and their generalizations as BCMP from Baskett, Chandy, Muntz, and  Palacios~\cite{baskett;chandy;muntz;palacios:75} and Kelly~\cite{kelly:76} networks, for a short review see Daduna~\cite{daduna:01a}. The resulting product form calculus provides closed form solutions for the most important performance metrics.

If the problem setting enforces to deviate from the necessary properties needed to hold for  using
product form calculus (e.g.\ exponential distributions, independence), often no closed analytical results for performance and reliability analysis exist, and therefore various approximation methods are
developed. A survey is the monograph by Bolch, Greiner, de Meer, and  Trivedi~\cite{bolch;greiner;demeer;trivedi:06}

Easier access to the field is via the textbook of
Gautam~\cite{gautam:12}. 

The algorithms described in these books are mostly developed by researchers from the field of computer and communications networks, but are claimed to be general purpose algorithms, e.g.\ to compute throughput of
any suitably defined network. Indeed, this has been proven in many applications, e.g.\ in production and logistics networks.

The topic of our paper is located in a rather different area:
A particular model from open-pit mining had to be analyzed. To be more precise: We are interested in the annual capacity of a (large) shovel in open-pit mining and, in a second step, in the number of trucks needed to optimally run the system.
From the very beginning, experience of
engineers in this field excluded product form models from being realistic. This suggests to apply one or more of the mentioned general purpose approximation algorithms at hand.

A comparison with results obtained by detailed simulation revealed that these algorithms often
do not perform well in this special application. Because of the high values  to invest
 the question  arises whether
it is possible to develop a specialized algorithm which can provide reliable performance predictions  before investment decisions are made. Indeed, for this particular case a heuristic
approximation from Dietrich Stoyan \& Helga Stoyan (1971) is at hand. This algorithm was developed for the special problem in pit-mining and related systems' analysis and is simple.
We revisit this algorithm here because it seems to be not accessible to the international community. It turned out that with today's computing systems there are no runtime problems, which holds for the mentioned general purpose
algorithms as well. Therefore we are only interested in precision, which is here defined by the distance from simulation outcomes
of performance metrics of interest obtained by either the Stoyan \& Stoyan (1971) algorithm or by the
general purpose algorithms.
It was a little bit surprising that despite of its simplicity the algorithm  outperforms in a realistic parameter setting which is characterized by relatively moderate variability in the processing process
all general purpose algorithms we tackled. 

To be honest we show that with high variability in the system direct application of this new algorithm is not recommendable. We will discuss this in detail and find out that our observation is in line with recommendations for to apply queueing models for performance evaluation in open-pit mining systems.

In a second step we therefore modify the simple algorithm to overcome this drawback. It turns out that the modified version performs well even in situations with high variability in the system.

The message of the paper therefore is: Although there exists in the computer science literature a variety of general purpose algorithms for performance evaluation of complex networks, it is often advisable
to look for special purpose algorithms which are adapted to the special problem dealt with.
This recommendation surely applies when the machines to buy or to construct are of very high value.

\paragraph{Some related work}
An introduction into the field of shovel-truck type operations is given by Carmichael~\cite{carmichael:86}  with an emphasis on ``How to apply queueing models''. He discusses the whole range of problems arising with queueing network models in this application area and gives recommendations how to
proceed  in such studies. Especially, he discusses data sets from the literature. The cyclic queues which are in the focus of our paper are the starting point of his description under the heading ``Reconciliation of theory and practice''. 

A more detailed description of closed queueing network models applicable in shovel-truck systems, especially of generalized Gordon-Newell networks and their algorithmic evaluation is given by Kappas and Yegulalp~\cite{kappas;yegulalp:91}. 

Ta, Ingolfsson, and Doucette~\cite{ta;ingolfsson;doucette:13} develop a linear integer program to optimize the number of trucks in a multi-shovel system with prescribed number of shovels. To determine the idle probabilities of the shovels they use simple, approximate finite source queueing models. 

Zhang and Wang~\cite{zhang;wang:09} consider a cyclic shovel-truck system of four stations:
Loading, traveling loaded, unloading, traveling  back empty,  where the unloading station is given special attention. By simulation they confirm that complexity of this station can be reduced to a single queue,
which allows to apply a general purpose algorithm from the literature to determine the system's capacity. 

For general principles of  modeling, analysis, and calculations in shovel-truck systems we refer to
the books Carmichael~\cite{carmichael:87} and Czaplicki~\cite{czaplicki:08}. For detailed information  on the closed two-station tandem queues which will be in the center of our paper we refer to Stoyan~\cite{stoyan:78} and Daduna~\cite{daduna:86a} 

For the general performance analysis algorithms we refer to standard literature, for example \cite{bolch;greiner;demeer;trivedi:06}  and~\cite{gautam:12}.

\paragraph{Structure of the paper}
In \prettyref{sect:DescriptionProblem} we describe the application area and point out the connection
of the transportation problem with loading and unloading to cyclic queueing networks.
Some more details on cyclic queueing networks are provided in  \prettyref{sect:CyclicQueue}.
In \prettyref{sect:SpecialStructure} we describe in detail the approach of Stoyan \& Stoyan to evaluate the main performance metric ``annual capacity of the shovel'' (which can be reduced to determine the stationary   idling probability of the shovel) and lay the ground for the algorithm to be invented next.
In \prettyref{sect:GeneralAlgEvaluation} we discuss comparison of several algorithms from the literature with the new algorithm which is given in \prettyref{sect:AlgStSt}.
The comparison is provided in \prettyref{sect:comparison} and it turns out that the algorithm of
Stoyan \& Stoyan performs well in the original realistic parameter setting.
In \prettyref{sect:DisturbLoading} we modify the parameter setting in a way that high variability
due breakdown interruptions of the shovel occurs in the system.
For completeness of the presentation we attach an appendix and describe in  \prettyref{sect:SurveyGenPurpose} the relevant class of Gordon-Newell networks and the general purpose
algorithms. We furthermore collect in this section omitted proofs and add related information which is helpful to understand the algorithms. 

Throughout the paper we discuss in detail the underlying modeling assumptions.

\paragraph{Remarks about pictures}
To the extent possible under law,
Ruslan Krenzler
has waived all copyright and related or neighboring rights to Figures
\ref{fig:queueing-model},
\ref{fig:mining-problem}, and \ref{fig:2-nodes-queueing-model}.\\
See \url{http://creativecommons.org/publicdomain/zero/1.0/}

\section{ Description of the problem  and the  model}
\label{sect:DescriptionProblem}

\paragraph{Structure of the system}
We consider a closed cyclic queueing network of four nodes,  numbered $\Jset := \{1,2,3,4\}$. See \Vref{fig:queueing-model}.
 The nodes are visited in this order sequentially by all  $K\geq 1$ customers in the network.

 Nodes $1$ and $3$ are single servers with unlimited waiting rooms, nodes $2$ and $4$ are infinite servers.
Service times at node $j$ have distribution function  $F_j$ with finite mean
$\mmu_j$ and finite variance $\sigma_j^2$. All service times are independent. We denote by $X_j$ a typical random variable distributed according $F_j, j=1,\dots,4$.
{Summarizing}: A customer starting at node $1$ visits sequences of nodes
\[\dots \longrightarrow \cdot/G/1/\infty\longrightarrow \cdot/G/\infty\longrightarrow
\cdot/G/1/\infty\longrightarrow \cdot/G/\infty\longrightarrow\dots .\]

\begin{figure}
\begin{center}
\includegraphics{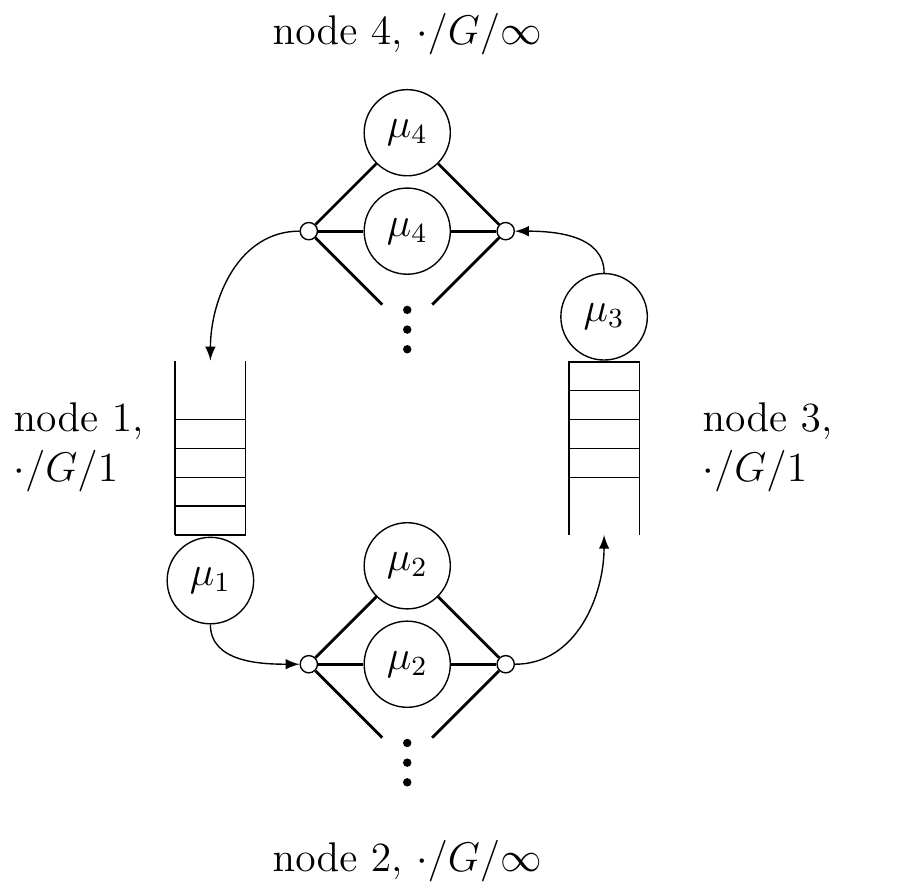}
\end{center}
\caption{Queueing model.\label{fig:queueing-model}}
\end{figure}

\subsection{Application area and details of the mining system}
\label{sect:ApplicationArea}
Our interest in this closed network comes from an important application in surface mining where
a system of four sequentially ordered components occur. Node $1$ stands for a shovel, which loads
trucks, the $K$ customers (jobs) stand for trucks (which are considered as identical), service at
node $2$ stands for  traveling  of loaded trucks, node $3$ for
an unloading system at a crusher and service at node $4$ for traveling back  of empty
trucks.  See \Vref{fig:mining-problem}.
Stoyan and Stoyan~\cite{stoyan;stoyan:71}  considered an analogous system,
where instead of trucks trains were used for material transport and instead of
a shovel a bucket-chain or bucket-wheel excavator was employed corresponding to the condition in the East German lignite mining. (This only
leads to different service time distributions.)

In the concrete example below we consider trucks with rated payloads of about \SI{220}{\tonne} and shovels with bucket volumes of about \SI{40}{\cubic\meter}. The price of a corresponding
truck is about \SI{6.5}{M\$}, that of a shovel about \SI{13}{M\$}.

\begin{figure}[ht]
\begin{center}
\includegraphics{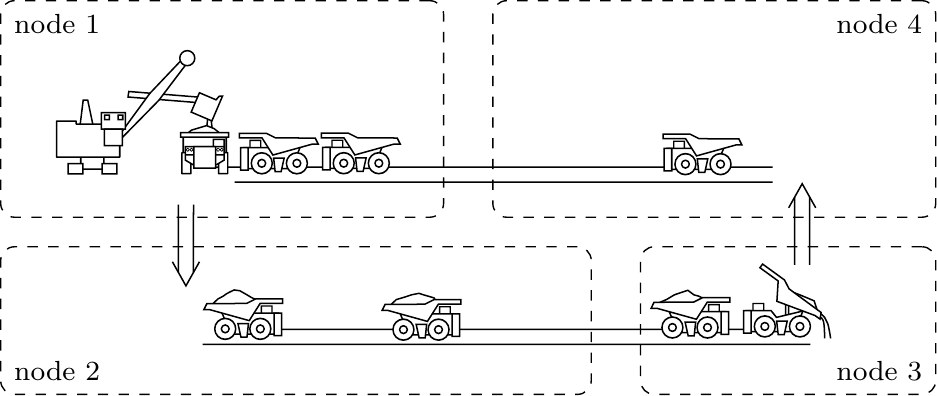}
\end{center}
\caption{Mining problem.\label{fig:mining-problem}}
\end{figure}

The assumption that node $3$ is of type $\cdot/G/1/\infty$ is a
simplification, because in a realistic model the unload system and crusher  encompasses a tandem system of subsequent  processing facilities  and  buffers.
This simplification  is acceptable in light of the paper of
Zhang and Wang~\cite{zhang;wang:09}. In the mining system investigated there, at the crusher node
there are discharge platforms, intermediate buffers, and a hopper. The crusher eventually feeds
a chain of belt conveyors. Zhang and Wang~\cite{zhang;wang:09}
showed by simulation that under ``normal conditions'' (the number of trucks $K$ not too large)
waiting times at this node can be neglected.
This will be utilized in the construction of the approximation algorithm.

We note that the operations in the mining system are even more complex.
The shovel moves from time to time, which changes the travel times of trucks.
Furthermore, the work of the shovel is often interrupted by repair times (planned or after
disturbances) and operating delays such as maneuvering, wall scaling or pad clean-ups.
In the following, disturbances of the shovel are integrated into truck loading times,
while disturbances of trucks and of the unloading process at node 3 are ignored,
similar to the procedure by Zhang and Wang~\cite{zhang;wang:09}, where  even disturbances
of the shovel are neglected, and in other papers on similar
mining transport systems,  e.g.,\ Peng, Zhang, and Xi~\cite{peng;zhang;xi:88},
Muduli~\cite{muduli:97},~\cite{ta;ingolfsson;doucette:13}.
This may be justified by employment of reserve trucks and of a large hopper. 

In many mining systems more than one shovel works. This  leads to
the need of employment of a dispatcher, who has to control the assignment
of empty trucks to the shovels. This will not be considered in detail here.

\paragraph{Target of our investigation}
The most important characteristic of the system is the annual capacity of the shovel.
To calculate this we need $\pi_1(0)$, the stationary probability that node $1$ is idle.
We will show that this can be obtained by  determining the throughput  of the shovel, measured in
the number of trucks loaded per time unit.
Throughout, in this paper time unit is 1 minute.

\subsection{Specifying the service time distributions}
\label{sect:ServiceTimeDistribution}

The service time distributions at the  nodes are determined by the technical conditions in surface mining
systems. Various authors have studied these  distributions  statistically, and in this paper we use the simplest and most popular distributions.

The service time at node $1$ is the load time for one truck including a possible repair time
in case of a disturbance (e.g.\ breakdown and repair)  during loading. The load time follows
in good approximation a normal distribution $F_1 = {\cal N}(\mmu_1,\sigma_1^2)$
as shown, e.g., by  Chanda and Hardy~\cite{chanda;hardy:11},
 Czaplicki~\cite{czaplicki:08},
Knights and Paton~\cite{knights;paton:10}, as a sum of (a random number) of random bucket cycle times.

The service times at nodes $2$ and $4$ are the travel times of trucks.
We assume that these times are normally distributed with means
$\mmu_2$ and $\mmu_4$ and variances $\sigma_2^2$ and $\sigma_4^2$.
The values for node $2$ are larger than for node $4$. 

We note that in the literature also other distributions have been discussed,
in particular the inverse Gauss distribution, see Stoyan and Stoyan~\cite{stoyan;stoyan:71}
 and  Panagiotou~\cite{panagiotou:93}, and Erlang distributions~\cite{ta;ingolfsson;doucette:13, carmichael:86}. It should be mentioned that Carmichael reports ``Best fit Erlang distributions for available published field data''~\cite[Table 2]{carmichael:86}, the number of exponential phases is usually
of an order that the shape of the  Erlang densities strongly resemble those of normal densities.

To obtain bounds it is recommended to use either deterministic service times (minimal coefficient of
variation) or exponentially distributed service times (coefficient of
variation $= 1$) for high variability in realistic scenarios,~\cite{ta;ingolfsson;doucette:13, carmichael:86}.

Also the service time at node $3$ (the crusher station) is assumed to be
normally distributed, with mean $\mmu_3$ and variance $\sigma_3^2$.

{\bfseries Discussion.}
As indicated above usage of normal distributions follows experiences from the
specific application in surface mining.
In our present investigation this usage is supported by the fact that the
numerical algorithms which we will apply in our performance analysis rely on (at most) two-parameter approximations. Prescribing expectation and variance of a distribution on $\mathbb R$
and taking the distribution with maximum entropy leads to usage of normal distributions.
It will turn out that the realistic parameter spaces in our applications with small disturbances, listed in \Vref{tab:ServiceTimes}, yield
probabilities for negative values of order $10^{-5}$. However adding large disturbances, described in \Cref{sect:DisturbLoading},  changes the negative probability significantly: It becomes $\approx0.292$ at node 1 and therefore the assumption of a normal distribution becomes questionable.
 
\section{The cyclic queueing model}
\label{sect:CyclicQueue}
The cyclic shovel-crusher-transport system under consideration  admits a Markovian modeling by counting the number of customers at each node and additionally
recording at nodes $1$ and $3$ the residual service time (loading, resp.\ unloading time for the truck in service, if any)
and recording at nodes $2$ and $4$ the residual service times (residual travel times) of all customers
in the infinite server queues.
The resulting Markov process has a unique stationary and limiting distribution.
This is easy to show if we have service time distribution which are finite mixtures of Erlang distributions.
Such distributions are dense (with respect to weak convergence) in the set of all distributions on
$[0,\infty)$ with finite expected values, Schassberger~\cite[p. 32, Satz 1]{schassberger:73}. The limiting arguments (continuity of queues) which establish the general statement can be found in Barbour~\cite{barbour:76}.

We assume throughout that this Markov process is in its stationary state.
Our interest is in  performance metrics of the stationary system, especially in $\pi_1(0)$,
the stationary probability  that node $1$  is empty.
It is easily seen that for this queueing network in general there is
no closed form analytical solution for obtaining stationary or asymptotic characteristics.

The standard approach to get information about the system's behavior
therefore usually has been simulation, in our problem setting with focus on the utilization of node $1$, measured in its idle times. Because we will discuss the power of analytical and direct numerical
procedures the following statement will be useful because many of the standard algorithms from
performance analysis are focused on computing mean values of queue lengths, waiting times, and especially
on throughput. No distributional assumptions are needed in the proposition.

\begin{prop}\label{prop:pi0}
Denote in the network with $K$ customers by $\lambda(K)$ the stationary throughput, measured in truck loads per time unit of the cyclic network with $K$ customers, i.e.\ the expected total  number
of departures in the cycle per time unit and $\lambda_j(K)$ the throughput of node $j$, i.e.\ the expected number of departures from $j$ per time unit.
\begin{ignorechecktex}
\begin{enumerate}[label=\emph{(\alph*)}]
\item\label{enu:prop-pi-0-local-lambda-is}
Then $\lambda_j(K) = \frac{1}{4}\lambda(K)$ holds.
\item\label{enu:prop-pi-0-pi-0-is}
The steady state probability for node $1$ to be empty is
\begin{equation}
	\pi_1^{(K)}(0) = 1 - \lambda_1(K)\cdot {\mmu_1}.\label{eq:pi0}
\end{equation}
\end{enumerate}
\end{ignorechecktex}
\end{prop}

\subsection{The cycle with product form steady state}\label{sect:PF}
\paragraph{Exponential service time at nodes 1 and 3}
If we  take service times at nodes  $1$ and $3$ as exponentially distributed, i.e.\ we
set the service times there only on the basis of a one-parameter approximation $F_j = \exp(\mu_j)$, $j=1, 3$ (which stems from
entropy arguments as above for normal distributions), we enter the field of product form networks and are able to write down explicitly the steady state distribution of the system in explicit and simple form, see e.g.~\cite{jackson:57, gordon;newell:67, baskett;chandy;muntz;palacios:75,kelly:76}.
To make the paper self-contained we sketch Gordon-Newell networks in \prettyref{sect:GN}.
It should be noted that even with these simple formulas problems occur with numerical evaluation of the
relevant probabilities.
There exist a variety of algorithms, developed in the area of performance analysis
of computer and communications networks  which output main characteristics of the networks
or parts of this, e.g., throughputs, mean queue lengths, mean travel times,
see~\cite{bolch;greiner;demeer;trivedi:06}.

\subsection{The cycle without product form steady state}
\label{sect:NonPF}
\paragraph{Non-exponential service time at nodes 1 and 3 under FCFS\protect\footnote{first-come, first-served}}
In this case no explicit steady state distribution has been found up to now.
Moreover, the network under consideration therefore does not meet the requirements needed to
apply the algorithms  based on  product form equilibrium.
In the area of computer systems and telecommunications networks  approximative algorithms are
developed over the last thirty years which transform the performance evaluation algorithms for product form networks into similarly structured approximation algorithms for non-product-form systems,
see~\cite{bolch;greiner;demeer;trivedi:06}.

On the other hand, it is often observed that these networks show  robustness against deviations from underlying distributional assumptions. This would justify to apply even the algorithms
based on product form assumptions to the non-product-form shovel-transportation-crusher network.

\section{Special structure of the mining system}
\label{sect:SpecialStructure}
The  structure of the system is  linear  within one cycle of a truck.
Such networks are prominent examples of models for teletraffic and transmission systems. But the observed
normal distributions for loading and unloading times  in the mining cycle are uncommon in,
 say, teletraffic models.
This raises the question whether there exist  ``special purpose algorithms'' to solve the specific
problem under consideration: To determine the throughput of the shovel in the mining cycle.
Astonishingly enough, the answer is to the positive.
Stoyan \& Stoyan~\cite[Section 8.3]{stoyan;stoyan:71}
proposed a short and efficient algorithm to determine the stationary idling property of the shovel
and then to apply \prettyref{prop:pi0}.

Their proposal utilized the fact that for realistic parameter values
$(\mmu_j,\sigma_j^2)$, $j=1,2,3, 4$, the waiting time at node $3$ (unloading at the crusher) can be neglected, see the discussion in  \prettyref{sect:ApplicationArea} and in the references mentioned there.
A consequence is that node $3$ {\bfseries acts as if it would be an infinite server} node $\cdot/G/\infty$.
This implies that we can replace the node sequence $2\to 3\to 4$ by a {\bfseries single node} of type
$\cdot/G/\infty$ with service time distribution $F_2*F_3*F_4$ (where $*$ denotes convolution)
exploiting  the overall independence assumptions. $F_2*F_3*F_4$ is therefore the distribution function of the so-called backcycle times,~\cite[p. 166]{carmichael:86}.

This reduces the model to a two-stage cycle\label{page:4to2-stage}
\[\dots\to\cdot/G/1/\infty\to \cdot/G/\infty\to\dots\]
See \Vref{fig:2-nodes-queueing-model}.
\begin{figure}[h]
\begin{center}
\includegraphics{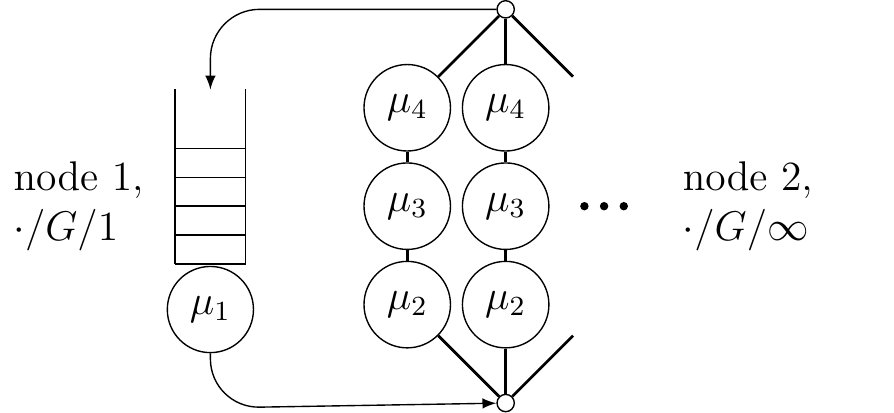}
\end{center}
\caption{Mining problem.\label{fig:2-nodes-queueing-model}}
\end{figure}

These closed two-stage tandem networks with one station being an infinite server are known as finite source queues (or ``repairman model''),
see Gross and Harris~\cite[Section~3.6]{gross;harris:74}.
$K$ is in that context the number of sources which send out customers to the other server (repairman).The authors
in~\cite{ta;ingolfsson;doucette:13} refer to these class of models when investigating idle times of
the shovel  in an oil sand mining operation system.
They fitted Erlang distributions to the empirical distribution functions of service times at the shovel and of the back-cycle time of the trucks.

For more detailed results concerning closed queueing network, especially cyclic queues, in truck-shovel systems for mining operations, see~\cite{carmichael:86}
and Kappas and Yegulalp \cite{kappas;yegulalp:91}.
Carmichel discusses the pros and cons for using the 4-stage cycle versus the 2-stage
cycle~\hbox{\cite[p. 162]{carmichael:86}}.


\subsection{The approximation of Stoyan \& Stoyan~\texorpdfstring{\cite{stoyan;stoyan:71}}{}}\label{sect:StoyanApproximation}
We consider the substitute two-stage cyclic network consisting of  a $\cdot/G/1$ node (for
the shovel) and a $\cdot/G/\infty$ node (for  the backcycle = transport + unloading + return time of the truck).
For concise notation we denote a typical  service time at node 1  by $S_{\bullet}$ or $S$, and
a typical idle time at node 1  by $I_{\bullet}$ or $I$.
The number of jobs (trucks) is $K$. We assume that the system is in its stationary state.

\paragraph{The main idea}
At node $1$ one observes an alternating sequence of service times and
idle times.  The latter can be zero with positive probability. This is observed when a job was waiting for service at the end of previous service.
Clearly, it holds
\begin{equation}\label{eq:pi10-1}
\pi_1(0) = \frac{E I}{E I + E S}.
\end{equation}
Thus we
have to determine the mean idle time $E I$.
For this, we tag a customer with label $1$. 
When his service at node $1$ expires we start a clock.
We stop the clock when $K-1$ further customers  at node $1$ are served. 
Then the time indicated by the clock hand is composed of a sequence of $K-1$ blocks, each block consisting of an idle time $I_k$ and a service time  $S_k$, $k=2,3,\dots,K$.

Let $T$ denote the backcycle time, i.e.\ the time from the departure of customer $1$ from node $1$ until 
he returns to node $1$ and enters there either the tail of the queue or starts immediately to be served.
Consider the random variable
\begin{equation}\label{eq:W-1}
 W := T - (I_2 + S_2 + \cdots + I_K + S_K).
\end{equation}
It follows
\begin{equation} \label{eq:W-2}
E W = E T - (K - 1)(E I + E S).
\end{equation}
Roughly, we want to compare the return time $T$ of customer $1$ to node $1$ with the time to empty node $1$ from the other $K-1$ customers.

If the cyclic two-stage network is {\bfseries overtake-free} the customer labeled $1$ is the uniquely
determined first one to enter service at node $1$ after all other $K-1$  customers have
departed there. If this is the case we can conclude
that $W$ refers uniquely to quantities dedicated to the distinguished customer $1$.
We have two cases:
\begin{ignorechecktex}
\begin{enumerate}[label=(\emph{\roman*})]
 \item
 If $W$ is positive,
it is the idle time before service of job $1$ starts in the next cycle, so $I \sim W$,

\item  otherwise, if $W$ is negative, it is
the negative of the waiting time of job $1$ before its next service commences and there is no idling of $1$
in this case,  so $I = 0$.
\end{enumerate}
\end{ignorechecktex}
Because customers are indistinguishable and from stationarity  thus
\begin{equation}
\label{eq:I-1}
 E I = E\max\{0, W\}.
\end{equation}

We assume that $W$ has a normal distribution with mean $\mmu_W$ and
variance $\sigma^2_W$. An approximative normal distribution is plausible since $W$ is a sum of many random variables, where the
majority of them are independent.
The normality assumption implies
\begin{equation} \label{Phi}
E I = E\max\{0, W\} =:i(\mmu_W) = \mmu_W
\Phi\left(\frac{\mmu_W}{\sigma_W}\right) + \sigma_W \varphi\left(\frac{\mmu_W}{\sigma_W}\right), 
\end{equation}
where $\Phi$ denotes the distribution function of standard normal
distribution and $\varphi$ the corresponding density function.

Neglecting the variance of the idle times in~\eqref{eq:W-1}, resp.~\eqref{eq:W-2},
for computing the variance of $W$ we use the approximation
\begin{equation*}
\sigma^2_W = var T + (K - 1)\cdot var S.
\end{equation*}
The mean $\mmu_W$ is chosen so that
\begin{equation*}
\mmu_W = E T - (K - 1)(E I + E S)
\end{equation*}
holds, where $E I$ is given by (\ref{Phi}).
Thus we have for $\mmu_W$ the equation 
\begin{equation*}
\mmu_W = E T - (K - 1)\left(E S +
\mmu_W \Phi(\frac{\mmu_W}{\sigma_W}) + \sigma_W \varphi\left(\frac{\mmu_W}{\sigma_W}\right)\right).
\end{equation*}

Consequently, with $\mmu_W$ at hand and with
$i(\mmu_W)$ given by (\ref{Phi}), we obtain
\begin{equation*}
\pi_1(0) = \frac{i(\mmu_W)}{i(\mmu_W) + E S}.
\end{equation*}

\paragraph{Discussion}
Seemingly, an essential point is that during one cycle of a customer (truck)  the network is topologically
overtake-free. Although the transportation nodes $\cdot/G/\infty$  are not overtake-free, the small variances of the transportation times (which is reasonable in the context), together with the FCFS regime of the loading and unloading station makes the cycle nearly overtake-free.
 Then the central formula~\eqref{eq:W-1}
\begin{equation*}
 W := T - (I_2 + S_2 + \cdots + I_K + S_K),
\end{equation*}
enters  with high precision although it is not completely true.
Overtake-freeness will be  exact by physical reasons, if railway transportation with only one railway line is used.\\
An important case where overtake-freeness holds is the system with deterministic service and travel times.
Despite of its simplicity it is of value as bounding system as Carmichel~\cite{carmichael:86} remarked.
He mentioned that this model was already used  by
Boyse and Warn~\cite{boyse;warn:75} as a simple approach to evaluate computing systems.
We shall discuss this later on and will provide in the appendix more details for this model.
Clearly, such flow approximations are of value in general network systems, and we shall therefore exploit a flow approximation to construct a general purpose algorithm as well.

\section{Algorithmic evaluation}\label{sect:GeneralAlgEvaluation}
The network under consideration in general does not meet the requirements needed to show a product
form equilibrium. The reason is that the shovel and the unloading facility work on a FCFS basis and the loading and unloading times for the transportation units are not exponentially distributed.
On the other hand, it is often observed that these networks are  robust against deviations from underlying distributional assumptions. This justifies to apply the algorithms
tailored for product form networks to the shovel-transportation-crusher network, see the MVA below.\\
Moreover, in the area of computer systems and telecommunications networks approximative algorithms are
developed over the last thirty years by transforming the performance evaluation algorithms for product form networks into similarly structured algorithms for non product form systems.
These are natural candidates for determining the idling probability of node $1$.

In the light of the present investigation we can characterize all these exact and approximative
algorithms as ``general purpose algorithms''. This general purpose methodology
has been applied in various  fields of applications of OR, e.g.\  production and transportation.

Our aim is to compare the precision of several algorithms when computing the idling probability of the shovel in the mining system. Precision is assessed by comparison with extensive simulations. The focus is on the question whether an extremely simple algorithm tailored for the specific mining system can outperform the
well established general purpose algorithms developed for general models in performance analysis for computer and communications.

We describe in \prettyref{sect:AlgStSt}\label{page:ListAlg}
\begin{enumerate}[label=$\langle\arabic*\rangle$, start=0]
\item\label{enu:alg-list-stoyan}
  ST\&ST, the special algorithm based on the approximation developed by
Stoyan \& Stoyan~\cite{stoyan;stoyan:71} from \prettyref{sect:StoyanApproximation}.
\end{enumerate}

Thereafter, the following general purpose algorithms (described in  \prettyref{sect:SurveyGenPurpose})
 will be  compared with this special algorithm in  \prettyref{sect:comparison}:\\
\begin{enumerate}[label=$\langle\arabic*\rangle$, resume]
\item\label{enu:alg-list-MVA}
MVA (Mean Value Analysis) for product form networks in the setting of \prettyref{sect:PF}.\\
Then we remove the assumption of exponential loading and unloading times (described in \prettyref{sect:NonPF}) and obtain:
\item\label{enu:alg-list-GMVA}
GMVA (Generalized Mean Value Analysis), see~\cite[Section 7.1.4.2]{gautam:12};
\item\label{enu:alg-list-ESUM}
ESUM (Extended Summation Method), see~\cite[Section 10.78]{bolch;greiner;demeer;trivedi:06} for non-product form networks;
\item\label{enu:alg-list-EBOTT}
EBOTT (Extended Bottleneck Approximation) for non-product form networks, see~\cite[Section 10.88]{bolch;greiner;demeer;trivedi:06};
\item\label{enu:alg-list-flow}
FLOW (deterministic flow approximation), where we consider the system as deterministic dynamical system, see~\cite{boyse;warn:75}.
\end{enumerate}
In any case we end in~\ref{enu:alg-list-MVA}--\ref{enu:alg-list-flow} with an expression for the throughput of the
shovel, which can be directly transformed  into the sought idling probability using the formula
$    1 - \pi_1^{(K)}(0) = \lambda_1(K) {\mmu_1}$ from  \prettyref{prop:pi0}.

We point out that the general purpose algorithms are run for the detailed 4-stage cycles, whereas the
special algorithm of  Stoyan \& Stoyan~\cite{stoyan;stoyan:71} is applied to the 2-stage cycle.  The latter is surely a drawback for precision because we neglect possible waiting times at station $3$ (unloading).

\paragraph{Remark}
We employed the two-parameter characterization of normal distributions in Algorithm~\ref{enu:alg-list-stoyan}, which
suggests as a possible candidate for comparison the Queueing Network Analyser (QNA),
developed by  Ward Whitt and coworkers for open networks of queues.
Whitt provided in~\cite{whitt:84a}  a thorough investigation and discussion  of how to use the QNA for open networks as a tool for approximating closed networks. He recommended mainly to use the FPM method
(Fixed Population Mean), where an open network with mean total population $K$
(of the closed network) is used. Whitt comments that the FPM does not perform well when there are only a few nodes  as in our mining system~\cite[p. 1916 and Section 7.3]{whitt:84a}.

\subsection{Algorithm based on the approximation of Stoyan \& Stoyan}\label{sect:AlgStSt}
The algorithm elaborates on the following input data.\\
 $K$= number of trucks; and
mean value and variance
$(\mmu_1, \sigma^2_1)$  of loading time at shovel for one truck,
$(\mmu_2, \sigma^2_2)$  of transport time from shovel to unload station,
$(\mmu_3, \sigma^2_3)$  of unload time for one truck,
 $(\mmu_4, \sigma^2_4)$ of travel time from unload station to shovel.

\begin{algorithm}
\begin{algorithmic}[1]
\Function{ST\&ST}{}\Comment{Calculate $\pi_1(0)$ and $\lambda_1$}
\State $ES \gets \mmu_1$
\State $ET \gets \mmu_2+\mmu_3+\mmu_4$
\State $VarT \gets \sigma_2^2+\sigma_3^2+\sigma_4^2$
\State $\sigma_W^2 \gets VarT +(K-1)\sigma_1^2$
\State $\mmu_W \gets$ solution of
$\mmu_W = ET-(K-1)\left(\mmu_W \Phi \left(\frac{\mmu_W}{\sigma_W}\right)
+\sigma_W \varphi \left(\frac{\mmu_W}{\sigma_W}\right)+ES\right)$
 \State $i(\mmu_W)  \gets \mmu_W \Phi \left(
\frac{\mmu_W}{\sigma_W}\right)
+\sigma_W \varphi \left(
\frac{\mmu_W}{\sigma_W}\right)$
 \State $\pi_1(0)\gets \frac{i(\mmu_W)}{i(\mmu_W)+ES}$
 \State $\lambda_1 \gets (1-\pi_1(0))\mmu_1$
 \State \Return $\pi_1(0),\lambda_1$
\EndFunction
\end{algorithmic}
\end{algorithm}

\section{Comparison of the  algorithms}
\label{sect:comparison}
We compared the algorithms~\ref{enu:alg-list-stoyan}--\ref{enu:alg-list-flow} 
with various parameter settings.
Because of the small size of the system none of the algorithm showed problems with runtime or memory.
Our focus therefore is only  precision which is determined by comparison with extensive simulations.
The simulations were run for the 4-station cycle as described in  \prettyref{sect:CyclicQueue},
i.e.\ we allowed (rare) queueing at the unloading station (node $3$) and did not enforce here the
additional assumption of non-overtaking which is introduced in the Stoyan \& Stoyan approximation
as an additional burden on its precision.
Similarly, the general purpose algorithms~\ref{enu:alg-list-MVA}--\ref{enu:alg-list-flow}
are applied to the more detailed and more realistic 4-station cycle. So in both cases we enhanced  precision, for the simulation as well as for
the general purpose algorithms. We therefore emphasize that the special purpose algorithm~\ref{enu:alg-list-stoyan}
will suffer from the simplifying assumption that the complete back-cycle is modeled by one queue.
Nevertheless, under normal conditions it shows superior precision.

\subsection{The loading process suffers only from small disturbances}
\label{sect:NoDisturbLoading}
The following results are derived under a typical and realistic set of parameters
given in \Vref{tab:ServiceTimes}. The main characteristic of these parameters is that
possible disturbances or interruptions of the shovel process (e.g.\ a loader does minor clearing work at the pit face, refuelling of vehicles, etc.)  are small enough to be classified as short term and can be covered by increasing the variability of the service times at node $1$, for more detailed discussion see~\cite[p. 171]{carmichael:86}.  We already discussed in  \prettyref{sect:ServiceTimeDistribution}
the problem of variability of the loading time at the shovel and will discuss  further details in \prettyref{sect:DisturbLoading}.
\begin{table}[!h]
\begin{center}
\begin{tabular}{lrr}
 & mean $(\mmu_{\bullet})$ in minutes & coefficient of variation $(C_{\bullet})$\\
 \midrule
loading time & 1.5 & 0.25\\
travel time loaded & 6.0 & 0.2 \\
unloading time & 1.0 & 0.1\\
travel time unloaded &4.0 & 0.2
\end{tabular}
\caption{Mean and coefficient of variation of service times.\label{tab:ServiceTimes}}
\end{center}
\end{table}
With the parameters given in \Vref{tab:ServiceTimes} we have run the simulation and the algorithms for  systems with $1$ to $10$ trucks. \Vref{fig:P0-Values-NoD} shows the absolute values for $\pi_1^{(K)}(0)$ while \Vref{fig:P0-AbsError-NoD}
shows the absolute deviation between the idle probabilities obtained by the respective algorithms
and the simulation. The following conclusion can be drawn from the figures:
\begin{ignorechecktex}
\begin{enumerate}[label=(\roman*)]
\item
For truck numbers up to $4$ all approximation values are close to the values obtained in the simulation.
\item
The deterministic FLOW approximation (coefficient of variation of service times = $0$)
is always a lower bound for the simulated values while the MVA (coefficient of variation of service times = $1$) is always an upper bound. Moreover, our study confirms for up to $8$ trucks the observation of Carmichael~\cite[p.169]{carmichael:86} that the errors of these approximations are of same magnitude with opposite sign.
\item
The algorithm of Stoyan \& Stoyan outperforms in the range of $1$ to $8$ trucks all other algorithms,
and is for $9$ to $10$ trucks almost good as GMVA\@. The latter does not fit the simulation results for
less than $9$ trucks and is worst for truck numbers below $6$.
\end{enumerate}
\end{ignorechecktex}
So the overall conclusion is that we can recommend to use the special algorithm~\ref{enu:alg-list-stoyan}
as long as the mentioned conditions for small variability are met.
Carmichael~\cite[p.171]{carmichael:86} suggests  that interruptions which are shorter than half of the cycle time for one truck should be incorporated into the service time variation.

\paragraph{Remark}
The observation that the MVA approximation, which needs a robustness property of the system under deviation from exponential assumptions, is not very good, is in line with a recommendation of
Bolch et al.~\cite[pp. 488, 489]{bolch;greiner;demeer;trivedi:06} not to use this product form approximation for   open tandems under FCFS with non exponential service times.

\subsection{Large disturbances of the loading process}\label{sect:DisturbLoading}
\begin{ignorechecktex}
In this section we consider a situation were the bound suggested by Carmichael: ``The interruptions  are shorter than half of the cycle time for one truck,'' is not met. The modeling assumption for the
interrupt processes are as follows.\\
If interruptions occur which cannot be neglected
it is realistic to assume that disturbances at the shovel
appear during its work time according to a Poisson process of
intensity $\alpha$, as argued by Stoyan and Stoyan~\cite{stoyan;stoyan:71} and 
Carmichael~\cite{carmichael:87}. The duration of a repair time as a result of a disturbance follows
an exponential distribution with parameter $\beta$, as several
authors recommend, e.g. Stoyan and Stoyan~\cite{stoyan;stoyan:71} and
Peng et al.~\cite{peng;zhang;xi:88}, see \Vref{tab:ServiceTimesBreakdown} for a realistic setting.
\end{ignorechecktex}
\begin{table}[!h]
\center
\begin{tabular}{lrr}
 &mean in minutes\\
\midrule
up-time of shovel & $\alpha^{-1} = 300$\\
 repair-time of shovel & $\beta^{-1} = 30$\\
\end{tabular}
\caption{Means of up and down time}
\label{tab:ServiceTimesBreakdown}
\end{table}
\subsubsection{Large disturbances of the loading process: Direct algorithm ST\&ST}
\label{sect:DisturbLoadingDirectAlg}
In this section  we apply the six algorithms listed in \prettyref{sect:GeneralAlgEvaluation} on
page \pageref{page:ListAlg} to the system with large disturbances originating from breakdown and repair
interruption of the shovel.\\
Recall that the service time at the shovel (loading time) is ${\cal N}(\mmu_1,\sigma^2_1)$ distributed.
Take $S_1\sim {\cal N}(\mmu_1,\sigma^2_1)$ and $X\sim \expdist(\alpha)$
as a typical uptime, $Y\sim \expdist(\beta)$
as a typical repair time (down time of the shovel) Then the modified service time at node $1$ is
\begin{equation}\label{eq:ServiceModif}
S_1^{(m)} = 
\begin{cases}
S_1 & \text{if $S_1 < X$},\\
S_1 + Y &\text{if $S_1 \geq X$}.
\end{cases}
\end{equation}

We assume the random variables $S_1, X, Y$ to be independent.
The modified service times $S_1^{(m)} = S_1 + 1_{\{X<S_1\}}\cdot Y$ are not normally distributed.
To apply a two-parameter approximation   we obtain by direct evaluation
\begin{prop}\label{prop:ModifiedService}
\begin{align*}
E S_1^{(m)} &= \mmu_1 + \frac{1}{\beta} P(X< S_1)\\
&=
\mmu_1 + \frac{1}{\beta} \left(\Phi\left(\frac{\mmu_1}{\sigma_1}\right)
- \exp\left(\frac{\alpha^2\sigma_1^2}{2} - \alpha\mmu_1\right)
\Phi\left(\frac{\mmu_1}{\sigma_1}-\alpha \sigma_1 \right)\right)\\
Var S_1^{(m)} &=
\sigma_1^2 + 2 E(1_{\{X<S_1\}}\cdot S_1) \frac{2}{\beta^2} +
P(X< S_1)\left(\frac{2}{\beta^2} - \frac{\mmu_1}{\beta}-\frac{1}{\beta^2} P(X< S_1)\right)\\
\end{align*}
with
\begin{align*}
E(1_{\{X<S_1\}}\cdot S_1)&=
\Phi\left(\frac{\mmu_1}{\sigma_1}\right) (\mmu_1-\alpha\sigma_1^2)
+\varphi\left(\frac{\mmu_1}{\sigma_1}\right) \sigma_1\\
&\alignphantomspace - \exp\left(\frac{\alpha^2\sigma_1^2}{2} - \alpha\mmu_1\right) \\
&\alignphantomspace \cdot 
\left(\Phi\left(\frac{\mmu_1-\alpha\sigma_1^2}{\sigma_1}\right) (\mmu_1-\alpha\sigma_1^2)
+\varphi\left(\frac{\mmu_1-\alpha\sigma_1^2}{\sigma_1}\right) \sigma_1\ \right)
\end{align*}
\end{prop}

Following modeling assumptions of the engineering literature and the principles described in the
 previous sections we assume for the modified  service time $S_1^{(m)}$ that it is normally distributed, ${\cal N}(\mmu_1(m),\sigma_1^2(m))$, with parameters obtained in \prettyref{prop:ModifiedService}.
\begin{align}\label{eq:ModifiedParameterE}
\mmu_1(m) &= E S_1^{(m)}& \sigma_1^2(m) &=Var S_1^{(m)}
\end{align}

With the data from \Vref{tab:ServiceTimes} and \Vref{tab:ServiceTimesBreakdown} we obtain modified mean and variance for node $1$ as given in  \Vref{tab:ServiceTimesModified}.
\begin{table}[!h]
\qquad\qquad\qquad\begin{tabular}{lrr}
 & \\
\midrule
modified mean of shovel service time & $\mmu_1(m) = 1.649603$\\
modified variance of shovel service time& $\sigma_1^2(m) =  9.122382$\\
\end{tabular}
\caption{Parameters of modified service time}
\label{tab:ServiceTimesModified}
\end{table}
With these new parameters for the shovel service time and the old ones for the other service times we have run the simulation and the algorithms for  systems with $1$ to $10$ trucks.  \Vref{fig:P0-Values-WithD,tab:P0-Values-WithD} show the absolute values for $\pi_1^{(K)}(0)$ while  \Vref{fig:P0-AbsError-WithD,tab:P0-AbsError-WithD}
show the absolute deviation between the idle probabilities obtained by the respective algorithms
and the simulation. The following conclusion can be drawn from the figures:
\begin{ignorechecktex}
\begin{enumerate}[label=(\roman*)]
\item
The precision of the algorithm of ST\&ST decreased dramatically, and for more than $5$ trucks it is the worst.
\item
Again the deterministic FLOW approximation (coefficient of variation of service times = $0$)
is always a lower bound for the simulated values while  the MVA (coefficient of variation of service times = $1$) is always an upper bound.\\
\item
Astonishingly, the MVA outperforms all other algorithms up to $8$  trucks cycling. It is beaten by the deterministic FLOW  approximation for $9$ to $10$ trucks, but as seen from 
\Vref{fig:P0-AbsError-WithD} below $9$ trucks MVA precision  is strictly better than that of FLOW\@.
\end{enumerate}
\end{ignorechecktex}

A possible explanation for the good performance of MVA in this parameter setting may be found by comparing coefficients of variation at node $1$: For MVA it is $1$ from the definition of
exponential distributions, while for all other algorithms it is set approximately $2$ (exception: FLOW).

\paragraph{Remark}
  The data in  \Vref{tab:ServiceTimes} shows that  maximal cycle time of a truck is reached if the truck finds at nodes $1$ and $3$ on its arrival all other trucks in front there. This results in a cycle time of $\sim \SI{20}{\minute}$. Carmichael's recommended limit of the interrupt time for application of queueing models of the type considered here therefore is ``lower than $\SI{10}{\minute}$''. The mean interrupt time
in our example is $\SI{30}{\minute}$, see \Vref{tab:ServiceTimesBreakdown}.

\subsubsection{Large disturbances of the loading process: Modified algorithm ST\&ST-m}
\label{sect:DisturbLoadingModifiedAlg}
The discussion in \prettyref{sect:DisturbLoadingDirectAlg} of the poor behavior of
the algorithm~\ref{enu:alg-list-stoyan} based on the approximation developed by Stoyan \& Stoyan posed the question whether that algorithm can be modified in a way that its precision is enhanced  and its simplicity is preserved. A first hint on how to proceed is given
by Carmichael~\cite{carmichael:86}. He recommends in case of large disturbances to distinguish between times of normal usage for the shovel and times when the
shovel is out of order.
\begin{ignorechecktex}
\begin{enumerate}
\item
The out-of-order times should be excluded because during these times no contribution to the (annual) capacity of the system is possible.

\item
The capacity during times of normal usage for the shovel should then be evaluated by
the standard algorithms.

\item
This recommendation leads us to proceed as follows:\\
Compute by algorithm \ref{enu:alg-list-stoyan}
the annual capacity and reduce that value by the factor
\begin{equation}\label{eq:ReduceFactor0}
\psi = \frac{\text{up-time of shovel}}{\text{up-time of shovel + down-time of shovel}} = \frac{\alpha^{-1}}{\alpha^{-1}+\beta^{-1}}.
\end{equation}
\end{enumerate}
\end{ignorechecktex}
We performed the respective experiments which revealed that algorithm~\ref{enu:alg-list-stoyan}  modified in this way
cannot compete with the best general purpose algorithms.\\

A successful resolution of the problem is as follows. In performing
the modification~\eqref{eq:ReduceFactor0}
we neglect the fact that a breakdown of the shovel can only occur if it is busy. To overcome this
shortage we replace~\eqref{eq:ReduceFactor0} by the following factor $\Psi$ which is obtained by a rough regeneration argument which we apply to the system when out-of-order times cannot be neglected
and we can identify periods when the system is in normal usage.
\begin{enumerate}
  \item For any service at the shovel in normal usage we perform a Bernoulli experiment (independent of the history of the system, at the end of the service periods, say)  with success probability
  $p=P(S_1>X)$, see~\eqref{eq:ServiceModif}. The mean number of non-interrupted services is $1/p$.
  \item The mean time between two time instants when service expires without intermediate interrupt by breakdown (i.e.\ under normal usage) is
  $\lambda_1^{-1}$.
  \item The time until breakdown when starting in normal usage is $\lambda_1^{-1} \cdot 1/p$ and
  thereafter starts a down time of mean length $1/\beta$. The relevant cycle length is
   $\lambda_1^{-1} \cdot 1/p + 1/\beta$.
  \item The proportion of time during that cycle when the shovel is up (not necessarily productive because of idling)  is therefore
  \begin{equation*}
  \Psi = \frac{\lambda_1^{-1} \cdot 1/p }{\lambda_1^{-1} \cdot 1/p + 1/\beta}.
  \end{equation*}
  \item
  Because the shovel breaks down only when it is busy, during breakdown times the queue length at the shovel (server $1$) is positive, i.e.\ the down times do not contribute to the idle times of the shovel.
  So the overall idle probability for the shovel in case of large disturbances ($=\ell d$) is with
  $\pi_1^{(K)}(0)$ computed by algorithm~\ref{enu:alg-list-stoyan} for the system without disturbances (under normal conditions)
  \begin{equation}\label{eq:Psi-Modified}
  \pi_1^{(\ell d, K)}(0)= \pi_1^{(K)}(0) \cdot \Psi = \pi_1^{(K)}(0) \cdot\frac{\lambda_1^{-1} \cdot 1/p }{\lambda_1^{-1} \cdot 1/p + 1/\beta}.
  \end{equation}
\end{enumerate}
The formula \eqref{eq:Psi-Modified} can be applied to any network with any algorithm which provides idle probability
$\pi_1^{(\ell d, K)}(0)$ and throughput $\lambda_1$.
If these values where obtained using ST\&ST-algorithm
we call the algorithm:
\begin{enumerate}[label=$\langle0-m\rangle$,leftmargin=1.6cm]
\item\label{enu:alg-list-stoyan-m}
  ST\&ST-m modified ST\&ST for problems with large disturbances.
\end{enumerate}

\begin{algorithm}
\caption{Modified Stoyan \& Stoyan algorithm for large disturbances}\label{alg:stoyan-and-stoyan-m}
\begin{algorithmic}[1]
	\Function{ST\&ST-m}{$\alpha$, $\beta$}\Comment{Calculate $\pi_1^{(\ell d, K)}(0)$}
\State $p \gets \Phi\left(\frac{\mu}{\sigma}\right)-\exp\left(\frac{\alpha^{2}\sigma^{2}}{2}-\alpha\mu\right)\Phi\left(\frac{\mu}{\sigma}-\alpha\sigma\right)$.
	\State  $(\pi^{(K)}_1{(0)},\lambda_1) \gets ST\&ST()$
	\State $\Psi \gets \frac{\lambda_1^{-1} \cdot 1/p }{\lambda_1^{-1} \cdot 1/p + 1/\beta}$
	\State $\pi_1^{(\ell d, K)}(0)\gets\pi_1^{(K)}{(0)}\cdot\Psi$
 \State \Return $\pi_1^{(\ell d, K)}(0)$
\EndFunction
\end{algorithmic}
\end{algorithm}

We performed the respective experiments, i.e.\ the simulation, and algorithms~\ref{enu:alg-list-stoyan}--\ref{enu:alg-list-flow}, and a modified version~\ref{enu:alg-list-stoyan-m} which incorporates~\eqref{eq:Psi-Modified} plus~\ref{enu:alg-list-stoyan}, as described above.
The data are the same as in \prettyref{sect:DisturbLoadingDirectAlg}.
In \Vref{fig:P0-Values-WithD,tab:P0-Values-WithD} we report the respective values of $\pi_1^{(\ell d, K)}(0)$. We see that the modified  version~\ref{enu:alg-list-stoyan-m} of the algorithm based on the approximation of Stoyan \& Stoyan is
extremely close to the simulated values for all sizes of the truck fleet.
In \Vref{fig:P0-AbsError-WithD,tab:P0-AbsError-WithD}
 we see that indeed algorithm~\ref{enu:alg-list-stoyan-m} outperforms all other
algorithms for fleet sizes $<10$. ST\&ST~\ref{enu:alg-list-stoyan} is as good as~\ref{enu:alg-list-stoyan-m} for only $1$ and $2$ trucks.
Only for $1$ truck FLOW is as good as~\ref{enu:alg-list-stoyan-m} and it is slightly better than~\ref{enu:alg-list-stoyan-m} for $10$ trucks.

Summarizing, we can say that the simple modification makes
version~\ref{enu:alg-list-stoyan-m} 
of the algorithm
ST\&ST for the computation of annual capacity for
the shovel superior to all its competitors.

\paragraph{Remarks about numerical results}
For simulation we used a discrete-event based simulation written in 
\emph{python\/} using \emph{SimPy\/} version 3.0.5 \cite{simpy:v3.0.5}. The simulation starts with full node 1 and then runs only once for \num{1000000} time units.

We compared the results with \emph{JMT-Java Modelling Tools} version 0.9.1~\cite{jmt:v0.9.1}, the results  for simulation are similar.

When there is only one truck in the system, an exact idle probability for node 1 can be obtained: It is $1-{\mmu_1}/({\sum_{i=1}^4 \mmu_i})$. This formula can be used for consistency check of simulations and approximations.

\vspace{3cm}

\begin{center}
\includegraphics[draft=false]{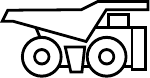}
\end{center} 


\newpage
\begin{figure}[H]
\center
\includegraphics{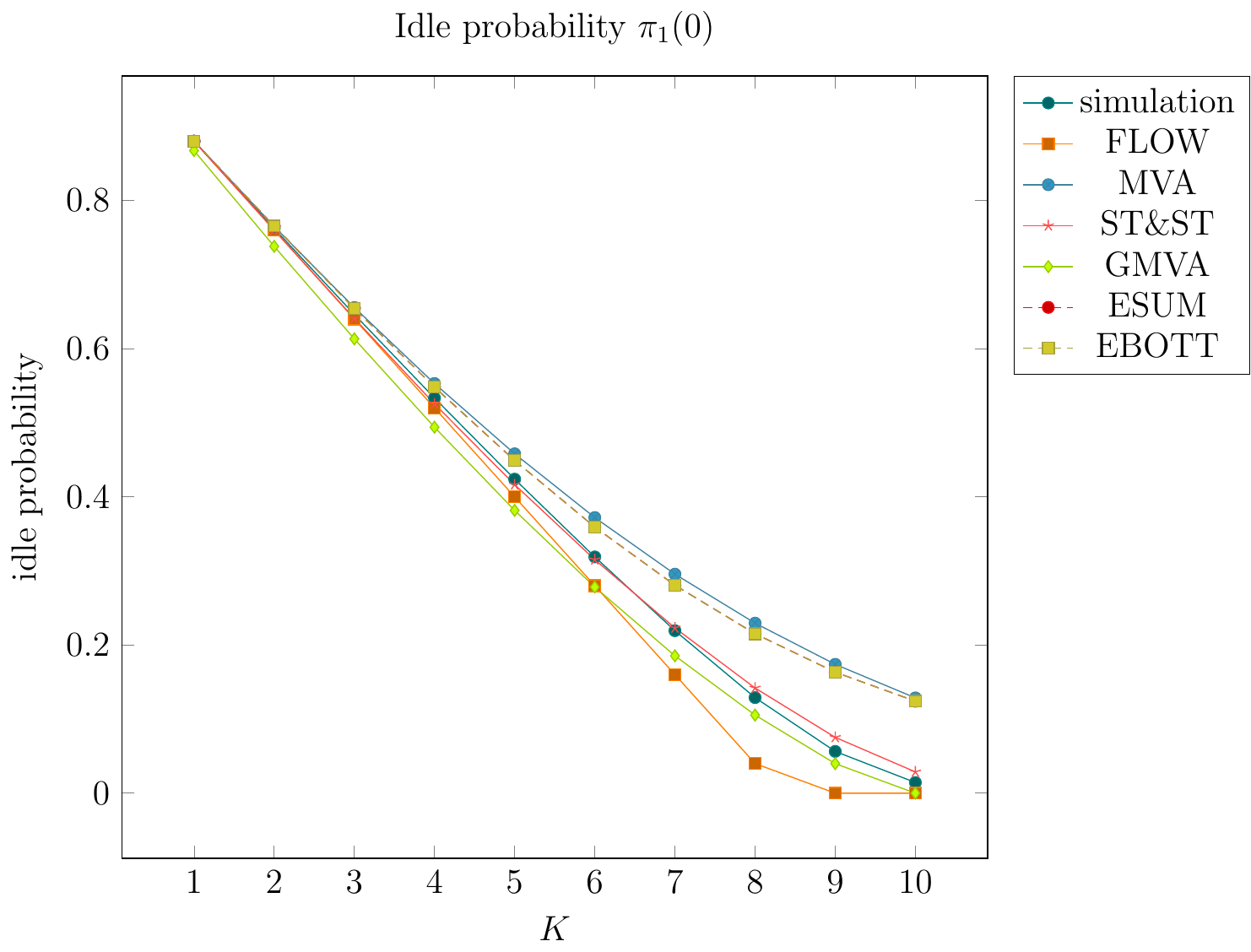}
\caption{Long-run idle probabilities for node 1 in a system with $K$ trucks $\pi_1^{(K)}(0)$, obtained through simulation and different approximation methods.\label{fig:P0-Values-NoD}}
\end{figure}

\begin{table}[H]
\center
\input{table-p0-values-no-d.tex}
\caption{Long-run idle probabilities for node 1 in a system with $K$ trucks. The values are rounded to three decimal places. The exact probability for $K=1$  is $0.880$.\label{tab:P0-Values-NoD}}
\end{table}

\begin{figure}[H]
\center
\includegraphics{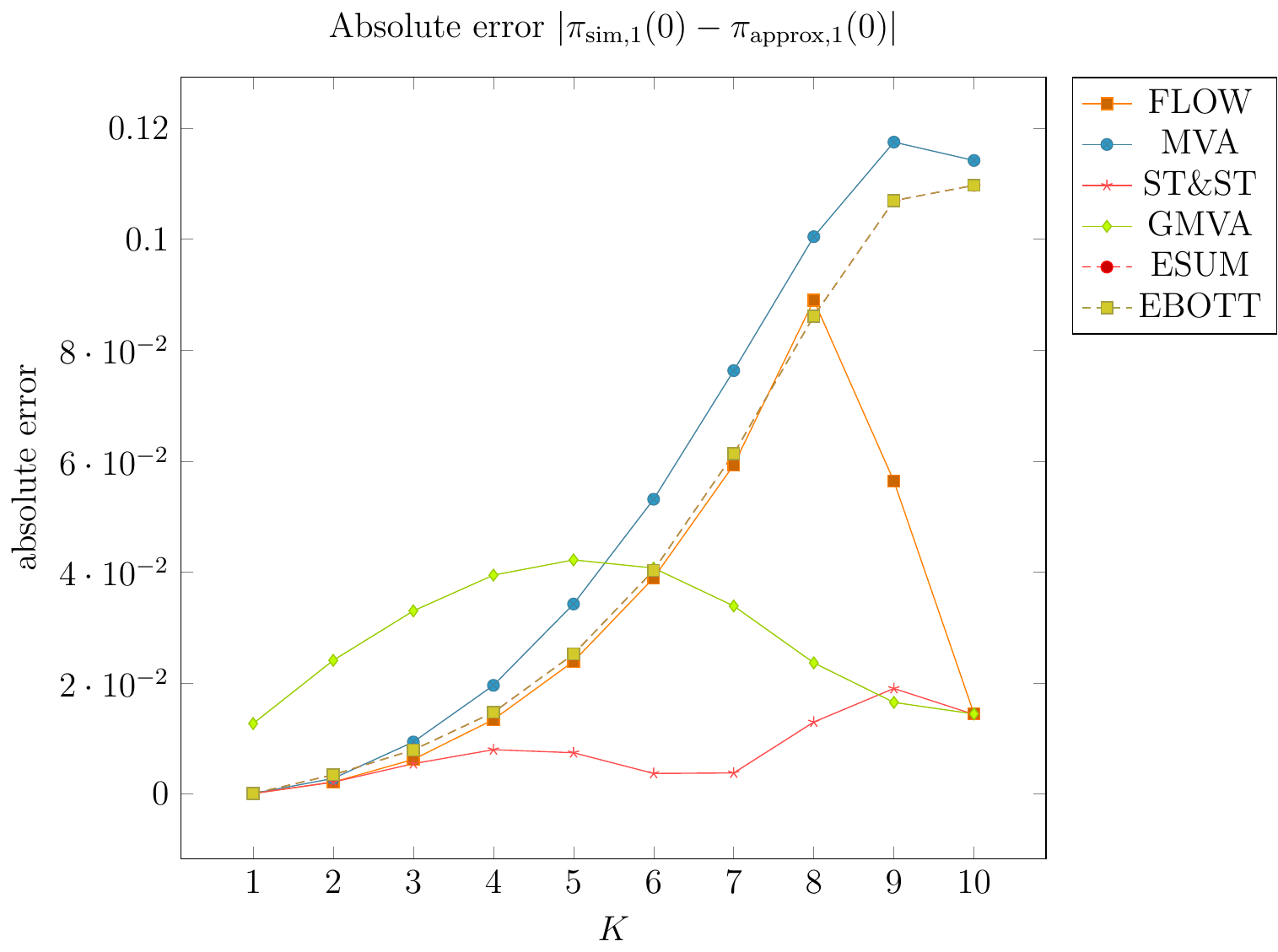}
\caption{\label{fig:P0-AbsError-NoD}}
\end{figure}

\begin{table}[H]
\center
\input{table-p0-abs-error-no-d.tex}
\caption{Absolute errors of approximation compared with simulation results. The errors are rounded to three decimal places.\label{tab:P0-AbsError-NoD}}
\end{table}

\begin{figure}[H]
\center
\includegraphics{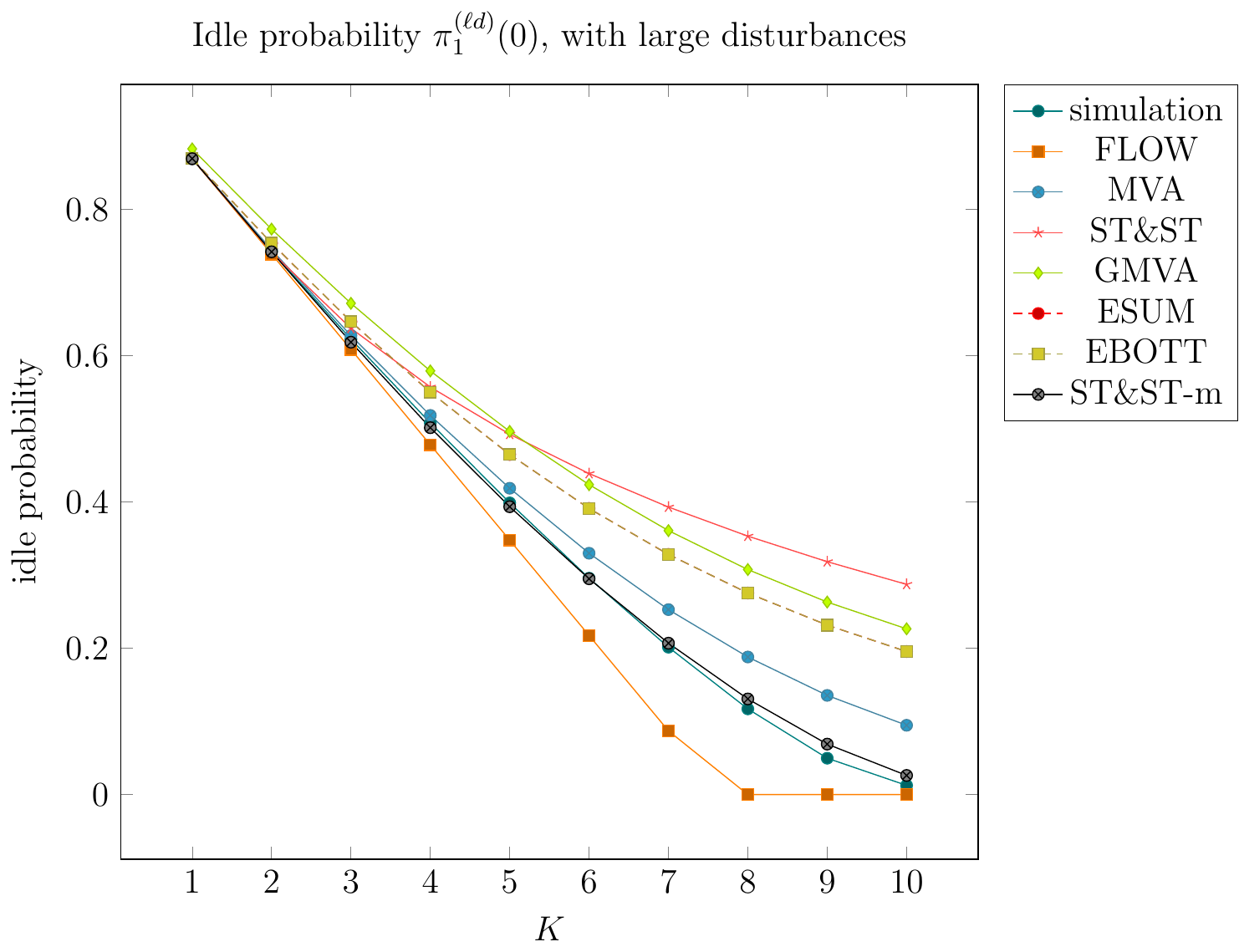}
\caption{Long-run idle probabilities for node 1 in  the model with large disturbances with $K$ trucks $\pi_1^{(\ell d, K)}(0)$, obtained through simulation and different approximation methods.\label{fig:P0-Values-WithD}}
\end{figure}

\begin{table}[H]
\center
\input{table-p0-values-with-d-1d30.tex}
\caption{Long-run idle probabilities for node 1 in the model with large disturbances with $K$ trucks. The values are rounded to three decimal places. The exact idle probability for $K=1$ is approximately $0.8695925$.\label{tab:P0-Values-WithD}}
\end{table}

\begin{figure}[H]
\center
\includegraphics{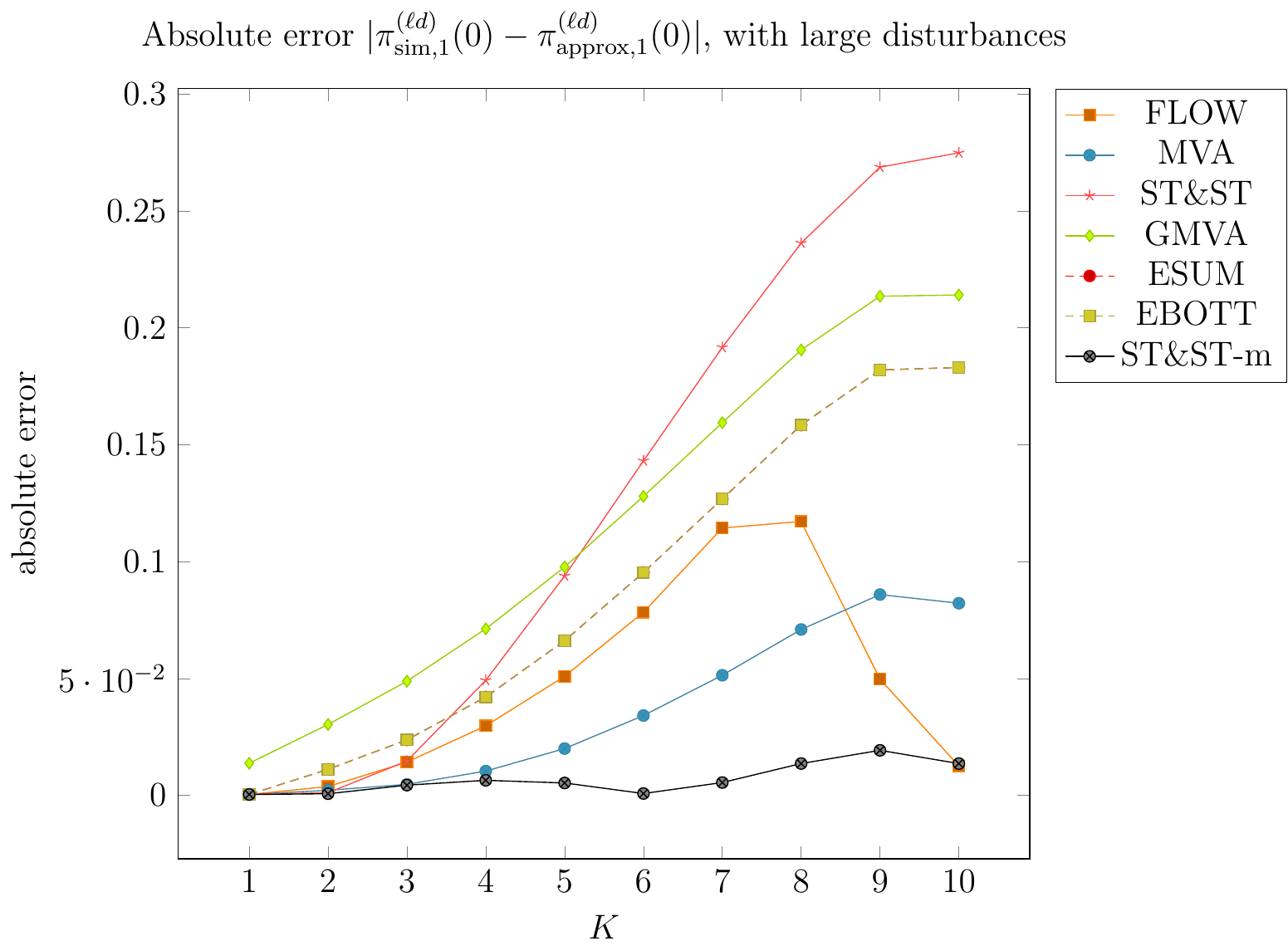}
\caption{\label{fig:P0-AbsError-WithD}}
\end{figure}

\begin{table}[H]
\center
\input{table-p0-abs-error-with-d-1d30.tex}
\caption{Absolute errors of approximation compared with simulation results in the model with large disturbances. The errors are rounded to three decimal places.\label{tab:P0-AbsError-WithD}}
\end{table}

\newpage
\appendix

\section{General purpose algorithms}\label{sect:SurveyGenPurpose}
\subsection{Gordon-Newell networks}\label{sect:GN}
A Gordon-Newell network~\cite{gordon;newell:67} with node set
$\Jset:= \{1,\dots,J\}$ has a fixed number $K$ of customers.
Customers are indistinguishable,
 follow the same rules, and request for  exponentially distributed service
 at all nodes. All these requests constitute an independent family
of variables. Nodes are exponential single servers with state
dependent service rates and  infinite waiting room under FCFS
regime. If at node $i$ there are $n_i>0$ customers, either
in service or waiting, service is provided there with intensity
$\mu_i(n_i)>0$.
Routing is Markovian,
a customer departing from node $i$ immediately proceeds to node
$j$ with probability $r{(i,j)}\geq 0$.
We assume that the  routing matrix
$r=(r(i,j):{i,j\in \Jset})$ is irreducible. Then
the traffic equations
\begin{equation} \label{eq:RS-traffic}
\eta_j=\sum_{i=1}^J \eta_i r(i,j),\quad\quad
j\in\Jset,
\end{equation}
have a unique stochastic solution, which we denote by $\eta = (\eta_j:j\in\Jset)$.

Note, that in~\cite{bolch;greiner;demeer;trivedi:06}, the solution of the traffic equation  \eqref{eq:RS-traffic} is normalized differently. Our particular choice for normalization of $\eta_j$ as a stochastic vector  does not influence the output of algorithms \labelcref{enu:alg-list-MVA,enu:alg-list-GMVA,enu:alg-list-ESUM,enu:alg-list-EBOTT}.

Development of the network is described by $X(t)=(X_1(t),\dots,X_J(t))$, the joint queue length process
 on state space
\[S(J,K) := \{\nvect =(n_1,\dots,n_j)\in \mathbb N^{\Jset}_0:n_1+\cdots+n_j =K \}.\]
$X(t)=(X_1(t),\dots,X_J(t))=(n_1,\dots,n_j)\in S(J,K)$ reads: at time $t$ there are
$X_j(t)=n_j$ customers present at node $j$, either in service or
waiting. The assumptions put on the system imply that $\X$ is a
ergodic  Markov process on state space $S(J,K)$.
The celebrated theorem of Gordon and Newell  states that the network process $\X$ has
the unique steady state and
limiting distribution $\xi$ on $S(J,K)$
with normalizing  constants $G(J,K)$
\begin{equation} \label{eq:GNSteadystate}
\xi(\nvect)=\xi(n_1,\dots,n_J) = {G(J,K)}^{-1}
\prod_{j=1}^{J} \prod_{\ell=1}^{n_j} \frac{\eta_j}{\mu_j(\ell)}, 
\qquad \nvect\in S(J,K).
\end{equation}
We consider here only  two types of servers, expressed in the rates $\mu_j(n_j)$.
A discussion why this restriction is adequate in open-pit mining and related fields
is given in~\cite[p. 46]{kappas;yegulalp:91}.
\begin{ignorechecktex}
\begin{enumerate}[label=(\arabic*.)]
\item
Single server nodes under FCFS, i.e.\ if $j$ is such a node, then $\mu_j(n_j) = \mu_j$
for all $n_j>0$, $\mu_j(0) =0$.

\item
Infinite server nodes, i.e.\ if $j$ is such a node, then $\mu_j(n_j) = \mu_j\cdot n_j$
for all $n_j>0$, $\mu_j(0) =0$.
\end{enumerate}
\end{ignorechecktex}

A more detailed survey with focus on application in open-pit mining and related areas is given in
\cite{kappas;yegulalp:91}.

{\bfseries Remark.}
From insensitivity theory for symmetric servers~\cite[Chapter 3.3]{kelly:79} it follows that for infinite server nodes we can substitute the exponential-$\mu_j$ service time distribution by any other distribution with the same mean $\mmu_j$
without changing the joint queue length distribution from~\eqref{eq:GNSteadystate}, although $X$ is no longer a Markov process.

\subsection{Short description of general purpose algorithms}\label{sect:ShortGenPurpose}
As indicated in \prettyref{sect:GeneralAlgEvaluation} we give now a brief survey of the
algorithms~\ref{enu:alg-list-MVA}--\ref{enu:alg-list-EBOTT}. For easier access to~\ref{enu:alg-list-ESUM}, ESUM, and~\ref{enu:alg-list-EBOTT}, EBOTT,
we will introduce these two by first describing the companion approximation algorithms for product form networks. Transfer
to the extended versions for non product form networks is then easy in both cases.

\subsubsection{Mean value analysis for product form networks}\label{sect:MVA}
Mean value analysis (MVA) was developed by Reiser and Lavenberg~\cite{reiser;lavenberg:80} for closed exponential Gordon-Newell networks. We give a sketchy description for the case of networks with
exponential single server nodes under FCFS and infinite server nodes with general service time distribution.
The service times at single server $j$ are exponential with mean $\mmu_j$, while
the service times at infinite server $i$ have general distribution function $F_i$ with mean
$\mmu_i$ and  variance $\sigma_i^2$.

The network consists of stations $\Jset :=\{1,2,\dots,J\}$ with $K\geq 1$ identical customers moving around according to a Markovian routing scheme $r:= (r(i,j):i,j \in \Jset)$. $p$ is irreducible with
unique steady state $\eta=(\eta_j:j \in \Jset)$.
The network admits a Markovian modeling by counting the number of customers at each node
 and additionally
recording at infinite server nodes the residual service time for every customer.
We denote by $X=((X_j(t):j\in\Jset):t\geq 0)$ the joint queue length process over all nodes
and add at infinite server node $j$ at time $t$ with $X_j(t)\geq 1$ the supplementary variables
$Y_j(t) = (Y_{j1}(t),Y_{j2}(t),\dots,Y_{j{X_j(t)}}(t))$. For $X_j(t)=0$ we set $Y_j(t) = 0$.
The network has a unique stationary distribution $\pi$.

The MVA computes recursively performance metrics of the network in steady state  with
increasing population sizes  $k=1,2, \dots,K$. We define mean values for population size $k$
with associated stationary distribution $\pi^{(k)}$
\begin{description}[leftmargin=!,labelwidth=\widthof{$\bar W_j(k)$:}]
\item[$\bar X_j(k)$:] expectation of $X_j(k)$, the number  of customers at node $j$ under $\pi^{(k)}$
\item[$\bar W_j(k)$:] expectation of   $W_j(k)$, the sojourn time of a customer at node $j$ under $\pi^{(k)}$
\item[$\lambda_j(k)$:] throughput at node $j$ $\equiv$ expected number of departures from node $j$
per time unit under $\pi^{(k)}$
\item[$\lambda(k)$:] total throughput of the network $\equiv$ expected number  of departures in the network per time unit under $\pi^{(k)}$
\end{description}
Note that $\lambda_j(k) =\eta_j\cdot \lambda(k)$ holds. The MVA recursion is

\begin{algorithm}[H]
\caption{Mean value analysis for product-form networks}\label{alg:mva}
\begin{algorithmic}[1]
\Function{MVA}{}\Comment{Calculate network throughput, average queue size and waiting time}
\State $\bar{X}_j(0) \gets 0$
\For{$k = 1 \ldots K$}
\State $\bar{W}_j(k) \gets
\begin{cases}
\mu_j^{-1} (1+\bar{X}_j(k-1)) & \text{if node } j \text{ is a single server node,} \\
\mu_j^{-1}  & \text{if node } j \text{ is an infinite server node.}
\end{cases} \label{line:mva-W}
$
\State $\lambda(k) \gets \frac{k}{\sum_j\in \Jset\eta_j W_j(k)}$\label{line:mva-lambda}
\State $\bar{X}_j(k) \gets \eta_j\lambda(k)W_j(k)$\label{line:mva-X}
\EndFor
\State \Return \parbox[t]{100mm}{$(\lambda(k) : k \in \{1,\ldots,K\})$,\\
$(\bar{X}_j(k) : k \in \{1,\ldots, K\}, j\in \Jset)$,\\
$(\bar{W}_j(k) : k \in \{1,\ldots,K\}, j\in \Jset)$}
\EndFunction
\end{algorithmic}
\end{algorithm}

\begin{ignorechecktex}
The step in line~\ref{line:mva-W} exploits the Arrival Theorem for Gordon-Newell
networks~\cite{lavenberg;reiser:80}, \cite{sevcik;mitrani:81}, while the steps
in lines~\ref{line:mva-lambda} and~\ref{line:mva-X}
are consequences of Little's  Theorem, see e.g.~\cite{stidham:72, stidham:74}.
\end{ignorechecktex}

\subsubsection{Generalized mean value analysis for non-product form networks}\label{sect:GMVA}
MVA as an exact algorithm breaks down when service times  at the single server nodes of the network
are not exponential, but are generally distributed according to some distribution $F_j$ at node $j$ with mean $\mmu_j$ and variance $\sigma_j^2$.

This network admits a Markovian modeling by counting the number of customers at each node
 and additionally
recording at single server nodes the residual service time for the job in service, if any, and
recording at infinite server nodes the residual service times  of the customers.
We denote by $X=((X_j(t):j\in\Jset):t\geq 0)$ the joint queue length process over all nodes,
add  at single server node $j$ at time $t$ with $X_j(t)\geq 1$ the supplementary variable
$Y_j(t)$. For $X_j(t)=0$ we set $Y_j(t) = 0$; for infinite server  nodes we proceed as in 
\prettyref{sect:MVA}

The network has a unique stationary distribution $\pi$.

While Little's Theorem is with correct interpretation still valid, the crucial point is that the Arrival Theorem does no longer hold. {\bfseries Assuming} that in the present setting the Arrival Theorem holds is a standard approximation in the literature which overcomes at least formally the arising difficulties.
The following Generalized  Mean Value Analysis (GMVA) is described in~\cite[Section 7.1.4.2]{gautam:12}, with a reference to~\cite{bolch;greiner;demeer;trivedi:06}. Gautam referred  to his scheme as
``MVA approximation'' for small population size.
With notation from \prettyref{sect:MVA} we have

\begin{algorithm}[H]
\caption{Generalized mean value analysis for non-product-form networks}\label{alg:gmva}
\begin{algorithmic}[1]
\Function{GMVA}{}\Comment{Calculate network throughput, average queue size and waiting time}
\State $\bar{X}_j(0) \gets 0$
\For{$k = 1 \ldots K$}
\State $\bar{W}_j(k) \gets
\begin{cases}
\mu_j^{-1} \left(\frac{1+C_{F_j}^2}{2}+\bar{X}_j(k-1)\right) & \text{if node } j \text{ is a single server node,} \\
\mu_j^{-1}  & \text{if node } j \text{ is an infinite server node.}
\end{cases} \label{line:gmva-W}
$
\State $\lambda(k) \gets \frac{k}{\sum_{j\in \Jset}\eta_j W_j(k)}$\label{line:mva-lambda-GMV}
\State $\bar{X}_j(k) \gets \eta_j\lambda(k)W_j(k)$\label{line:mva-X-GMV}
\EndFor
\State \Return \parbox[t]{100mm}{$(\lambda (k) : k \in \{1,\ldots,K\})$,\\
$(\bar{X}_j(k) : k \in \{1,\ldots, K\}, j\in \Jset)$,\\
$(\bar{W}_j(k) : k \in \{1,\ldots,K\}, j\in \Jset)$}
\EndFunction
\end{algorithmic}
\end{algorithm}

The step in
\ref{line:gmva-W}
 takes the Arrival Theorem to be true similar as for Gordon-Newell
networks~\cite{lavenberg;reiser:80},~\cite{sevcik;mitrani:81} with an additional correction term.
This term adjusts for the fact that an arriving customer finding a single server
with service time distribution $F_j$ busy, sees in stationary state
a mean residual service time of size
\begin{equation}\label{eq:MRST1}
\frac{1 + C^2_{F_j}}{2},\quad \text{with squared coefficient of variation}~
C^2_{F_j}= \frac{\sigma_j^2}{\mu_j^{-2}}~\text{of the service time.}
\end{equation}
The steps in~\ref{line:mva-lambda-GMV} and~\ref{line:mva-X-GMV}
are consequences of Little's Theorem, see e.g.~\cite{stidham:72, stidham:74}.
\subsubsection{Summation method for product form networks}\label{sect:SUM}
The summation method is an approximation for product form networks, our description
follows~\cite[Section 9.2]{bolch;greiner;demeer;trivedi:06}. Its advantage is that
no recursion in the number of
customers is necessary.
With notation from \prettyref{sect:MVA} (deleting population sizes $(k)$) we define for throughput
$\lambda_i$ at node $i$ with utilization $\rho_i:=\lambda_i/\mu_i$
\begin{equation}\label{eq:fi1}
f_i(\lambda_i) =
    \left\{
      \begin{array}{ll}
        \frac{\rho_i}{1 - \frac{K-1}{K}\rho_i}, & \hbox{if node $i$ is a single server;} \\
        \frac{\lambda_i}{\mu_i}, & \hbox{if node $i$ is an infinite server.}
      \end{array}
    \right\} \approx\bar{X}_i.
\end{equation}
Remark: In~\cite[Section 9.2]{bolch;greiner;demeer;trivedi:06} the $f_i$ are given for exponential
multi-server nodes with $m_i>1$ service channels as well. The respective formula does \emph{not\/} boil
down to the single-server case given above, which is taken
from~\cite{bolch;greiner;demeer;trivedi:06} as well.

The functions $f_i(\lambda_i)$ are non-decreasing in $\lambda_i$ and the $\lambda_i$
are defined in the range
\begin{eqnarray*}
   &0\leq \lambda_i \leq \mu_i,& \quad\text{if node $i$ is a single server;} \\
   &0\leq \lambda_i \leq K\cdot\mu_i,&  \quad\text{if node $i$ is an infinite server}.
\end{eqnarray*}
Because $\lambda_i =\eta_i\cdot\lambda$ we can define
\begin{equation}\label{eq:glambda}
    g(\lambda) := \sum_{i=1}^J f_i(\lambda_i) \approx \sum_{i=1}^J  \bar X_i  = K,
\end{equation}
where the last equality follows from the fixed population constraint. The summation algorithm then is

\begin{algorithm}[H]
\caption{Summation method for product-form networks.}\label{alg:sum}
\begin{algorithmic}[1]
\Function{SUM}{$\epsilon$}\Comment{Calculate network throughput.}
\State $\lambda^{(l)} \gets 0$ \Comment{Chose lower bound for $\lambda$.}
\State $s_i \gets
\begin{cases}
m_i & \text{if node }i \text{ has } m_i<\infty \text{ servers} \\
K & \text{if node } i \text{ is an infinite server}
\end{cases}$
\State $\lambda^{(u)} \gets \min_i \left\{\frac{\mu_i s_i}{\eta_i}\right\}$
 \Comment{Chose upper bound for $\lambda$.}
\Repeat \Comment{Determine $\lambda$ by bisection algorithm}
\State $\lambda \gets \frac{\lambda^{(l)} + \lambda^{(u)}}{2}$
\State $\lambda_i \gets \lambda \cdot \eta_i$
\State $g(\lambda) \gets \sum_{i=1}^K f_i(\lambda_i)$ where the $f_i$ are defined in~\eqref{eq:fi1}.
\If{$g(\lambda)>K+\epsilon$}
\State $\lambda^{(u)} \gets \lambda$
\Else
\State $\lambda^{(l)} \gets \lambda$ \Comment{Effectively if $g<K-\epsilon$}
\EndIf
 \Until{$|g(\lambda)-K|\leq \epsilon$}
\State \Return \parbox[t]{100mm}{$(\lambda_j : j \in \Jset)$}

\EndFunction
\end{algorithmic}
\end{algorithm}

\begin{remark}
\label{rem:SUM-loc}If $\lambda$ is found then in the computation step we set approximately $\lambda_i=\eta_i\cdot\lambda$  according to the definition of local throughput and
$\bar X_i =f_i(\lambda_i)$ by using the approximation in~\eqref{eq:fi1} again, and consequently
by Little's Theorem which applies here, $\bar W_i =f_i(\lambda_i)/\lambda_i$.
\end{remark}
\subsubsection{Extended summation method for non-product form networks}\label{sect:ESUM}
Extending the summation method to networks with infinite servers and single servers under FCFS
with non-exponential service time is an easy task now: While for the infinite server $i$ we use again
$f_i(\lambda_i) =\rho_i$ we  replace for single server $i$ the $f_i(\lambda_i)(\approx \bar X_i)$.
In~\cite[(10.78), p. 503 and (10.88), p. 505]{bolch;greiner;demeer;trivedi:06} it is suggested to use
with $a_i:= (1+C^2_{F_i})/2$
\begin{equation}\label{eq:fi2}
f_i(\lambda_i) =
    \left\{
      \begin{array}{ll}
        \rho_i +\frac{\rho_i^2\cdot a_i}{1 - \frac{K-1- a_i}{K-1}\rho_i}, & \hbox{if node $i$ is a single server;} \\
        \rho_i, & \hbox{if node $i$ is an infinite server.}
      \end{array}
    \right\} \approx\bar{X}_i.
\end{equation}
With this substitute then the algorithm in \prettyref{sect:SUM} is run,
and \prettyref{rem:SUM-loc} applies here as well.
\begin{remark}\label{rem:Kompatible-1}
The formula for $f_i$ in case of single servers with non-exponential service times (in~\eqref{eq:fi2})
is in case of exponential service times \emph{not\/} the first formula in~\eqref{eq:fi1}. Because we apply the
extended summation method only for single server nodes with non-exponential service times we follow the recommendation in~\cite[(10.88), p. 505]{bolch;greiner;demeer;trivedi:06}.
\end{remark}
\subsubsection{Bottleneck approximation method for product form networks}\label{sect:BOTT}
The bottleneck approximation method is a computational  approximation for product form networks,
our description follows~\cite[Section 9.3]{bolch;greiner;demeer;trivedi:06}.
Its advantage compared to the exact MVA is that no recursion in the number of customers is necessary.

We assume that there exists exactly one bottleneck node.
With notation from \prettyref{sect:MVA} (deleting population sizes $(k)$) we define for throughput
$\lambda_i$ at node $i$ with $\rho_i:=\lambda_i/\mu_i$
\begin{equation}\label{eq:fi3}
f_i(\lambda_i) =
    \left\{
      \begin{array}{ll}
        \frac{\rho_i}{1 - \frac{K-1}{K}\rho_i}, & \hbox{if node $i$ is a single server;} \\
        \frac{\lambda_i}{\mu_i}, & \hbox{if node $i$ is an infinite server.}
      \end{array}
    \right\} \approx\bar{X}_i.
\end{equation}
We define the inverse functions of the strictly increasing $f_i(\cdot)$ by $h_i(\cdot)$.
Because the bottleneck in our problem is in every case a single server or an infinite server, we fix it only for these cases~\cite[p. 449]{bolch;greiner;demeer;trivedi:06}:
\begin{equation}\label{eq:hi3}
h_i(\bar X_i) =
 \left\{
      \begin{array}{ll}
      \frac{\bar X_i}{1 + \frac{K-1}{K}\bar{X}_i}   & \hbox{if bottleneck node $i$ is a single server;} \\
       \frac{\bar X_i}{K}, & \hbox{if  bottleneck  node $i$ is an infinite server.}
      \end{array}
    \right\} \approx\rho_i.
\end{equation}
The pseudocode for the bottleneck approximation method is listed in \Vref{alg:bottleneck}.
\begin{algorithm}[h]
\caption{Bottleneck approximation for product-form networks.}\label{alg:bottleneck}
\begin{algorithmic}[1]
\Function{BOTT}{$\epsilon$}\Comment{Calculate network throughput and average queue size.}
\State $s_i \gets
\begin{cases}
1 & \text{if node }i \text{ is a single server} \\
K & \text{if node } i \text{ is an infinite server}
\end{cases}$
\State $\lambda \gets \min_i \left\{\frac{\mu_i s_i}{\eta_i}\right\}$
 \Comment{Chose initial $\lambda$.}
\State $\text{bott} \gets \argmin_i \left\{\frac{\mu_i s_i}{\eta_i}\right\}$\label{line:BOTT-determine-bottleneck}
\Comment{Chose initial bottleneck index.}
\Repeat
\State $\lambda_i \gets \lambda \cdot \eta_i$
\State $\rho_i \gets \lambda_i/(\mu_i\cdot s_i)$
\State
$\bar{X}_i \gets f_i(\lambda_i)$
\State $g(\lambda) \gets \sum_{i=1}^K f_i(\lambda_i)$ where the $f_i$ are defined in~\eqref{eq:fi3}.
\State $\bar{X}_{\text{bott}}\gets \bar{X}_{\text{bott}} \cdot \frac{K}{g(\lambda)}$
\State
$\rho_{\text{bott}}\gets h_{\text{bott}}(\bar{X}_{\text{bott}})$
\State
$\lambda \gets \frac{\rho_{\text{bott}}\cdot
s_{\text{bott}}\cdot
 \mu_{\text{bott}}}{\eta_{\text{bott}}}$
 \Until{$\left| \frac{K}{g(\lambda)} - 1\right| \leq \epsilon$}
 \State $\bar{X}_i \gets \bar{X}_i \cdot \left| \frac{K}{g(\lambda)} \right| $
\State \Return \parbox[t]{100mm}{$(\lambda_j : j \in \Jset)$,\\
$(\bar{X}_j: j\in \Jset)$}
\EndFunction
\end{algorithmic}
\end{algorithm}
\begin{remark}
By a coincidence, our system with \textbf{four} trucks  and mean service times from \Vref{tab:ServiceTimes} has two bottlenecks:
on the nodes $1$ and $4$:
$\frac{\mu_1 s_1}{\eta_1}=
\frac{1/1.5\cdot 1}{1/4} =
\frac{1/6.0\cdot 4}{1/4}=
\frac{\mu_4 s_4}{\eta_4}
$
Therefore  line~\ref{line:BOTT-determine-bottleneck} from the original
BOTT or EBOTT algorithm will not work properly. In our implementation of BOTT and EBOTT, if  line~\ref{line:BOTT-determine-bottleneck} returns multiple indexes, we just take the smallest one. This is adequate in our problem setting because node 4 is an infinite server.

In general this modification of the bottleneck approximation is not good, because the
systems with multiple bottlenecks can be very different from the system with a single bottleneck. In our case the modification seems to return consistent results.

\end{remark}

\subsubsection{Extended bottleneck approximation  method for non-product form networks}\label{sect:EBOTT}
Extending the bottleneck approximation  method is similar to extending the summation method to networks with infinite servers and single servers under FCFS
with non-exponential service time.
We assume again that there exists exactly one bottleneck node.

 While for the infinite server $i$ we use again
$f_i(\lambda_i) =\rho_i$ we  replace for single server $i$ the $f_i(\lambda_i)$.
In~\cite[(10.88), p. 505]{bolch;greiner;demeer;trivedi:06} it is suggested to use
with $a_i:= (1+C^2_{F_i})/2$  (note that \prettyref{rem:Kompatible-1} applies here as well)
\begin{equation*}
f_i(\lambda_i) =
    \left\{
      \begin{array}{ll}
        \rho_i +\frac{\rho_i^2\cdot a_i}{1 - \frac{K-1- a_i}{K-1}\rho_i}, & \hbox{if node $i$ is a single server;} \\
        \frac{\lambda_i}{\mu_i}, & \hbox{if node $i$ is an infinite server.}
      \end{array}
    \right\} \approx\bar{X}_i.
\end{equation*}
The problem is now to invert for the single server node the $f_i$ which in our case indeed is
needed because the bottleneck  in any case is such node.
With this substitute $h_i$ the algorithm in \prettyref{sect:BOTT} is run.

For $\cdot/G/1/\infty$ nodes under FCFS in~\cite[(10.90), p. 505]{bolch;greiner;demeer;trivedi:06} it is suggested to use (recall~\eqref{eq:MRST1}) with  stationary mean residual service time
\begin{equation*}
a_i := \frac{1 + C^2_{F_i}}{2},\quad \text{with squared coefficient of variation}~
C^2_{F_i}= \frac{\sigma_i^2}{\mu_i^{-2}}~\text{of the service time}
\end{equation*}
and with
\begin{equation*}
    b_i:= \frac{K-1-a_i}{K-1}
\end{equation*}
the approximated inversion formula
\begin{equation*}
h_i(\bar X_i) =
 \left\{
      \begin{array}{ll}
\frac{-(1+b_i\cdot\bar{X}_i) + \sqrt{{(1+b_i\cdot\bar{X}_i)}^2 + 4\cdot\bar{X}_i(a_i-b_i)}}{2(a_i-b_i)}  .
      & \parbox{3.5cm}{if bottleneck node $i$ is a single server;} \\
      \\
       \frac{\bar{X}_i}{K}, & \parbox{3.5cm}{if bottleneck  node $i$ is an infinite server.}
      \end{array}
    \right\} \approx\rho_i.
\end{equation*}
\subsubsection{Deterministic flow approximation}
\label{sect:FLOW} Flow approximations of stochastic systems
assume that all interarrival and service times in the cycle are deterministic. So, the lengths of the
service times in the cycle are $\mmu_j, j=1,2,3,4$.
We can assume  $\mmu_1>\mmu_3$ for our problem setting. 

Then after an initial transient phase there will be no more queueing at node $3$ and the backcycle times
$T$ (see \prettyref{sect:StoyanApproximation}) are deterministic  $T=\mmu_2 + \mmu_3 + \mmu_4$.
Additionally the backcycles of the customers do not interfere, which results in the observation that we can substitute the node sequence $(2,3,4)$ by a single infinite server node (or even by a single
$\cdot/D/K$ queue when evaluating e.g.\ cycle times, system throughput, idling probability of node $1$.
Saying the other way round, the approximative reduction of the 4-stage cycle to a 2-stage cycle by Stoyan \& Stoyan (see \prettyref{sect:SpecialStructure}, page~\pageref{page:4to2-stage}) is in this case (with respect to the mentioned
performance metrics) exact.

\begin{ignorechecktex}
The resulting cycle consisting of a $\cdot/D/\infty$ and a $\cdot/D/K$  node was used by Boyse and
Warn~\cite{boyse;warn:75} to propose ``a straightforward model for computer performance prediction''.They
investigate an interactive terminal-computing system where the 2-stage deterministic cycle models the paging mechanism for updating the memory. Translating the notation of~\cite{boyse;warn:75} into ours,
their relevant formula~\cite[(2-1) on p.~78]{boyse;warn:75} reads
\end{ignorechecktex}
\begin{equation}\label{eq:BW2-1}
    1- \pi^{(K)}_1(0) =
    \begin{cases}
        \frac{K\mmu_1}{\sum_{i=1}^4 \mmu_i}, & \text{if } (K-1)\mmu_1\leq \sum_{i=2}^4 \mmu_i;\\
        1, & \text{if }(K-1)\mmu_1 > \sum_{i=2}^4 \mmu_i.
    \end{cases}
\end{equation}

Note that we can write the dichotomy in~\eqref{eq:BW2-1} as 
$$\sum_{i=2}^4 \mmu_i -(K-1)\mmu_1 ~~\binom{\leq}{>}~~ 0.$$

The left side of this expression is exactly the right side of~\eqref{eq:W-1} which leads to the two cases in~\eqref{eq:I-1}. Similarly to~\eqref{eq:I-1},
a more compact formula can be obtained by introducing the asymptotic waiting time  $\bar V_1$ of the customers at node $1$
\begin{equation}
    \bar{V}_{1}  =\max\left(0,K\mmu_{1}-\sum_{i=1}^{4}\mmu_{i}\right),\label{eq:det-V-bar-10}
\end{equation}
and to  write the complementary probability of~\eqref{eq:BW2-1} as
\begin{equation}
\pi_{1}^{(K)}(0)=1-\frac{K\mmu_{1}}{\sum_{i=1}^{4}\mmu_{i}+\bar{V}_{1}}.\label{eq:pi-1-0-deterministic-10}
\end{equation}

\begin{algorithm}
\begin{algorithmic}
\Function{FLOW}{}\Comment{
\parbox[t]{70mm}{Calculate waiting time, idling probability  of node $1$ in a  deterministic system with $\mmu_1 > \mmu_3$}}
\State
$\bar{V}_{1} \gets \max\left(0,
\mmu_{1}K-\sum_{i=1}^{4}\mmu_{i}\right)$
\State
$\pi_{1}^{(K)}(0)\gets 1-\frac{K\mmu_{1}}{\sum_{i=1}^{4}\mmu_{i}+\bar{V}_{1}} $
\State
$\lambda_1 \gets (1-\pi_{1}^{(K)}(0))\mmu_{1}$
\State \Return \parbox[t]{100mm}{$\pi^{(K)}_1(0), \lambda_1 , \bar{V}_{1}$
}
\EndFunction
\end{algorithmic}
\end{algorithm}

\section{Omitted proofs and complements}\label{sect:ProofsComplements}
\subsection{Omitted proofs}\label{sect:Proofs}
\begin{proof}[Proof of \prettyref{prop:pi0}]
\ref{enu:prop-pi-0-local-lambda-is}  follows from the cyclic structure and the fact that in equilibrium all four nodes must have the same throughput. To prove
\ref{enu:prop-pi-0-pi-0-is}
we will apply Little's formula twice and elaborate on this with the pathwise proof of this formula by
Stidham~\cite{stidham:72,stidham:74}. Existence and uniqueness of the limiting and stationary distribution will
then transform the pathwise formulas into expectations and probabilities.
The general formula $L=\lambda\cdot W$ is interpreted for the relevant quantities of node $1$ firstly  as
\begin{itemize}
\item
{Mean queue length = Arrival rate $\times$ Mean sojourn time} \mbox{\hspace{0.5cm}}\\
and secondly as
\item {Mean number of waiting customers = Arrival rate $\times$ Mean waiting time}.\\
Subtracting we obtain
\item {Mean number of  customers in service =
Arrival rate $\times$ Mean service time},\\ which is $ 1 - \pi_1^{(K)}(0) = \lambda_1(K)\cdot  {\mmu_1}$.
\end{itemize}
\end{proof}

\emph{Proof\/} of \prettyref{prop:ModifiedService} will be given in a sequence of statements.

We have to analyze a random variable $S_1^{(m)}:=S_1+1_{\{X<S_1\}}Y$ with $S_1\sim\mathcal{N}(\mu,\sigma^{2})$,
$X\sim \expdist(\alpha)$ and $Y\sim \expdist(\beta)$. $S_1$, $X$ and $Y$
are  independent. $S_1$ describes a typical service
time without breakdown, $X$ is generated by the breakdown process and $Y$
describes repair time. $S_1^{(m)}$ describes a typical modified service time considering breakdowns
and repair. 

Recall that $\Phi$ is distribution function of the standard normal distribution, and
 $\phi$ the density function of $\Phi$. 
 
We assume that $P(S_1<0)$ is negligible.
$P(X<S_1)$ is the break down probability and $P(X\geq S_1)$
the probability that a  service will be completed without failures.

\begin{prop}\label{prop:break-down-proabability}
For  independent $S_1\sim N(\mu,\sigma)$, $X\sim \expdist(\alpha)$  holds
\[
P(X<S_1)=\Phi\left(\frac{\mu}{\sigma}\right)-\exp\left(\frac{\alpha^{2}\sigma^{2}}{2}-\alpha\mu\right)\Phi\left(\frac{\mu}{\sigma}-\alpha\sigma\right).
\]
\end{prop}
\begin{proof}
\begin{align}
P(S_1\leq X) & =\int P(\underbrace{S_1\leq t}_{\mathclap{ \text{st.\ ind.\ from $X$}}}|X=t)dF_{X}(t)\nonumber \\
 & =\int P(S_1\leq t)dF_{X}(t)=\int P(S_1\leq t)f_{X}(t)dt\nonumber \\
 & =\int_{0}^{\infty}F_{S_1}(t)\alpha\exp(-\alpha t)dt \nonumber\\
 & =-F_{S_1}(t)\exp(-\alpha t)dt|_{0}^{\infty}+\int_{0}^{\infty}f_{S_1}\exp(-\alpha t)dt\nonumber \\
 & =F_{S_1}(0)+\frac{1}{\sigma\sqrt{2\pi}}\int_{0}^{\infty}\exp\left(-\frac{1}{2}{\left(\frac{t-\mu}{\sigma}\right)}^2-\alpha t\right)dt.\label{eq:VP-B-less-X-proof-1}
\end{align}
In order to calculate the last integral we use
\begin{align}
-\frac{1}{2}{\left(\frac{t-\mu}{\sigma}\right)}^{2}-\alpha t & =-\frac{1}{2}\left({\left(\frac{t-\mu}{\sigma}\right)}^{2}+\underbrace{2\alpha\sigma\frac{(t-\mu)}{\sigma}+2\alpha\mu}_{=2\alpha t}+\alpha^{2}\sigma^{2}-\alpha^{2}\sigma^{2}\right)\nonumber \\
 & =-\frac{1}{2}{\left(\frac{t-\mu}{\sigma}+\alpha\sigma\right)}^{2}-\alpha\mu+\frac{\alpha^{2}\sigma^{2}}{2}\nonumber \\
 & =-\frac{1}{2}{\left(\frac{t-(\mu-\alpha{\sigma}^{2})}{\sigma}\right)}^{2}+\frac{\alpha^{2}\sigma^{2}}{2}-\alpha\mu\label{eq:VP-exp-mul-exp-transformation}
\end{align}
and transform the last integral of~\eqref{eq:VP-B-less-X-proof-1} into
\begin{align*}
 & \mathrel{\hphantom{=}}\frac{1}{\sigma\sqrt{2\pi}}\int_{0}^{\infty}\exp\left(-\frac{1}{2}{\left(\frac{t-\mu}{\sigma}\right)}^{2}-\alpha t\right)\\
 & =\exp\left(\frac{\alpha^{2}\sigma^{2}}{2}-\alpha\mu\right)\frac{1}{\sigma\sqrt{2\pi}}\int_{0}^{\infty}\exp\left(-\frac{1}{2}{\left(\frac{t-(\mu-\alpha\sigma^{2})}{\sigma}\right)}^{2}\right)dt.
\end{align*} $\int_{0}^{\infty}\exp\left(-\frac{1}{2}{\left(\frac{t-(\mu-\alpha\sigma^{2})}{\sigma}\right)}^{2}\right)dt$
is the probability $P(A\geq0)$ of a normal random variable $A$ with
mean $(\mu-\alpha\sigma^{2})$ and variance $\sigma$. Therefore it
holds
\[
P(S_1\leq X)=\underbrace{1-\Phi\left(\frac{-\mu}{\sigma}\right)}_{=\Phi(\mu/\sigma)}+\underbrace{\left(1-\Phi\left(\frac{-\mu}{\sigma}+\alpha\sigma\right)\right)}_{=\Phi(\mu/\sigma-\alpha\sigma)}\exp\left(\frac{\alpha^{2}\sigma^{2}}{2}-\alpha\mu\right).
\]
\end{proof}

In the following \prettyref{lem:VP-E-B-mu-1-X-less-B}
we will use
\begin{equation}
\label{eq:positive-normal}
\frac{1}{\sigma\sqrt{2\pi}}\int_{0}^{\infty}t\exp\left(-\frac{1}{2}{\left(\frac{t-y}{\sigma}\right)}^{2}\right)dt=\Phi\left(\frac{y}{\sigma}\right)y+\varphi\left(\frac{y}{\sigma}\right)\sigma.
\end{equation}
from {\cite[(3.5.4)]{stoyan;stoyan:71}}.

\begin{lemma}
\label{lem:VP-E-B-mu-1-X-less-B}For  independent $S_1\sim N(\mu,\sigma)$, $X\sim \expdist(\alpha)$ holds
with $y:=(\mu-\alpha\sigma^{2})$
\[
E(S_11_{\{X<S_1\}})=\Phi\left(\frac{\mu}{\sigma}\right)y+\varphi\left(\frac{\mu}{\sigma}\right)\sigma-\exp\left(\frac{\alpha^{2}\sigma^{2}}{2}-\alpha\mu\right)\left(\Phi\left(\frac{y}{\sigma}\right)y+\varphi\left(\frac{y}{\sigma}\right)\sigma\right).
\]
\end{lemma}
\begin{proof}
\begin{align*}
E(S_11_{\{X<S_1\}}) &=\int tP(\underbrace{X<t}_{\mathclap{\text{st.\ ind.\ from $S_1$}}}|S_1=t){dP}_{S_1}(t)\\
 & =\frac{1}{\sigma\sqrt{2\pi}}\int t(1-\exp(-\alpha t))1_{\{t\geq0\}}\exp\left(-\frac{1}{2}{\left(\frac{t-\mu}{\sigma}\right)}^{2}\right)dt\\
 & =\frac{1}{\sigma\sqrt{2\pi}}\int_{0}^{\infty}t\exp\left(-\frac{1}{2}{\left(\frac{t-\mu}{\sigma}\right)}^{2}\right)dt\\
 & -\frac{1}{\sigma\sqrt{2\pi}}\int_{0}^{\infty}t\exp(-\alpha t)\exp\left(-\frac{1}{2}{\left(\frac{t-\mu}{\sigma}\right)}^{2}\right)dt
\end{align*}

Using the~\eqref{eq:VP-exp-mul-exp-transformation} we get
\begin{align*}
 & \mathrel{\hphantom{=}}\frac{1}{\sigma\sqrt{2\pi}}\int_{0}^{\infty}t\exp(-\alpha t)\exp\left(-\frac{1}{2}{\left(\frac{t-\mu}{\sigma}\right)}^{2}\right)\\
 & =\exp\left(\frac{\alpha^{2}\sigma^{2}}{2}-\alpha\mu\right)\frac{1}{\sigma\sqrt{2\pi}}\int_{0}^{\infty}t\exp\left(-\frac{1}{2}{\left(\frac{t-(\mu-\alpha\sigma^{2})}{\sigma}\right)}^{2}\right)dt.
\end{align*}

From~\eqref{eq:positive-normal} it follows with
$y:=(\mu-\alpha\sigma^{2})$
\begin{align*}
E(S_11_{\{X<S_1\}}) & =\Phi\left(\frac{\mu}{\sigma}\right)y+\varphi\left(\frac{\mu}{\sigma}\right)\sigma\\
 & -\exp\left(\frac{\alpha^{2}\sigma^{2}}{2}-\alpha\mu\right)
\left(\Phi\left(\frac{y}{\sigma}\right)y+\varphi\left(\frac{y}{\sigma}\right)\sigma\right).
\end{align*}
\end{proof}

\begin{prop}
\label{prop:VP-E-Z}For the modified service time $S_1^{(m)}:=E(S_1+1_{\{X<S_1\}}Y)$ with
independent $S_1\sim N(\mu,\sigma)$, $X\sim \expdist(\alpha)$
and $Y\sim \expdist(\beta)$ holds
\begin{align}
 E(S_1^{(m)})  &=  \mu+\frac{1}{\beta}P(X<S_1),\quad\text{and}\label{eq:ModifiedE1}\\
 Var(S_1^{(m)})  &=  \sigma^{2}+2E(1_{\{X<S_1\}}S_1)\frac{2}{\beta^{2}}+
P(X<S_1)\left(\frac{2}{\beta^{2}}-\frac{2\mu}{\beta}-\frac{1}{\beta^{2}}P(X<S_1)\right).
\label{eq:ModifiedVar1}
\end{align}
\end{prop}
\begin{proof}
\eqref{eq:ModifiedE1} follows directly from
\begin{equation*}
E(S_1^{(m)})=E(S_1+1_{\{X<S_1\}}Y)=E(S_1)+E(1_{\{X<S_1\}})E(Y).
\end{equation*}

We have furthermore
\begin{align*}
E({(S_1^{(m)})}^{2}) & =E(S_1^{2}+2(1_{\{X<S_1\}}YS_1)+1_{\{X<S_1\}}^{2}Y^{2})\\
 & =E(S_1^{2})+2E(1_{\{X<S_1\}}S_1)E(Y)+E(1_{\{X<S_1\}})E(Y^{2})\\
 & =\sigma^{2}+\mu^{2}+2E(1_{\{X<S_1\}}S_1)\frac{1}{\beta}+P(X<S_1)\frac{2}{\beta^{2}},
\end{align*}
and therefore~\eqref{eq:ModifiedVar1} follows from
\begin{align*}
Var(S_1^{(m)}) & =E({(S_1^{(m)})}^{2})-E^{2}(S_1^{(m)})\\
 & =\sigma^{2}+\mu^{2}+2E(1_{\{X<S_1\}}S_1)\frac{2}{\beta^{2}}+P(X\leq S_1))\frac{2}{\beta^{2}}
 -{\left(\mu+\frac{1}{\beta}P(X<S_1)\right)}^{2}\\
 & =\sigma^{2}+2E(1_{\{X<S_1\}}S_1)\frac{2}{\beta^{2}}
 +P(X<S_1)\frac{2}{\beta^{2}}-\frac{2\mu}{\beta}P(X<S_1)-\frac{1}{\beta^{2}}{P(X<S_1)}^{2}\\
 & =\sigma^{2}+2E(1_{\{X<S_1\}}S_1)\frac{2}{\beta^{2}}
 +P(X<S_1)\left(\frac{2}{\beta^{2}}-\frac{2\mu}{\beta}-\frac{1}{\beta^{2}}P(X<S_1)\right).
\end{align*}
\end{proof}

\subsection{Details for the deterministic flow model}\label{sect:ComplementDeterministic}

We consider the deterministic 4-stage cycle and assume that $\mmu_{1}>\mmu_{3}$.
We will analyze some details of the long-time behavior
of the system.

An important property of the deterministic system
is that after all customers passed node $1$ once, the inter-arrival times at  node $3$ are greater
than service times by at least $(\mmu_{1}-\mmu_{3})$. Consequently , after a finite initial period, independent of the initial configuration
\begin{enumerate}
\item the waiting times at node 3 are always zero,
\item the arrival times at the node 1 are obtained from  departure times
from the node $1$ by adding a delay $\mmu_{2}+\mmu_{3}+\mmu_{3}$
(the backcycle  times),
\item the inter-arrival times at any node are greater or equal to $\mmu_{1}$.
\end{enumerate}
Henceforth we assume that the initial period has passed and
start the system when at time $0$ these properties are in force.

We count the customers  arriving
 to the station $1$.
We denote  the sequence of arrival times at station $1$ by
$(\gamma_1(n), n\geq 1)$, the sequence of departure times by $(\tau_1(n), n\geq 1)$.
The $n$-th customer's waiting time at $1$ is denoted $V_1(n)$,
and the inter-arrival time between the $n$-th and $n+1$-th customer by
$A_1(n)$. If the $n$-th arriving customer is served at all nodes and returns to the station
$1$ after passing the nodes $2, 3, 4 $ he generates the $n+K$-th arrival at $1$.
It holds
\begin{equation}
A_1(n)\geq\mmu_{1}\qquad\forall n\in\mathbb{N} \label{eq:det-A-n}
\end{equation}
\begin{equation}
\gamma_1(n+K)=\tau_1(n)+\mmu_{2}+\mmu_{3}+\mmu_{4}\qquad\forall n\in\mathbb{N}.  \label{eq:det-gamma-i+K}
\end{equation}
Furthermore, for any queuing system with a single server and deterministic
service times  $\mmu_{1}$ we have
\begin{equation}
\tau_1(n)=\gamma_1(n)+V_1(n)+\mmu_{1}\label{eq:det-tau-equation}
\end{equation}
and (the Lindley recursion)
\begin{equation}
V_1(n+1)=\max(0,V_1(n)+\mmu_{1}-A_1(n))\qquad\forall n\in\mathbb{N}.\label{eq:det-V-n+1}
\end{equation}

From~\eqref{eq:det-gamma-i+K} and~\eqref{eq:det-tau-equation} follows
for all $n\in\mathbb{N}$
\begin{align}
\gamma_1(n+K) & =\gamma_1(n)+\Big(V_1(n)+\mmu_{1}+\mmu_{2}+\mmu_{3}+\mmu_{4}\Big)\label{eq:det-gamma-i+K+from-gamma-i}\\
\Longleftrightarrow\sum_{i=n}^{\mathclap{n+K-1}}A_1(i)= & V_1(n)+\mmu_{1}+\mmu_{2}+\mmu_{3}+\mmu_{4}.\label{eq:det-sum-An-V-1}
\end{align}

From~\eqref{eq:det-A-n} and~\eqref{eq:det-V-n+1} follows
\begin{equation}
V_1(n+1)\leq V_1(n),\label{eq:det-decreasing-V-n}
\end{equation}
and
\begin{equation*}
V_1(n+1)=V_1(n)+\mmu_{1}-A_1(n)\qquad\text{if }V_1(n+1)>0. \label{eq:det-V-n-positive-recursion}
\end{equation*}

\eqref{eq:det-decreasing-V-n} implies that whenever the waiting
time at station $1$ becomes zero it will remain zero.
We can classify whether waiting times at station $1$ become zero and obtain a dichotomy.

\begin{prop}
\label{prop:det-not-saturated-and-saturation-point}For a cycle with
$K$ customers, such that
\begin{equation}
K\mmu_{1}\leq\mmu_{1}+\mmu_{2}+\mmu_{3}+\mmu_{4}\label{eq:no-waiting-time-conditions}
\end{equation}
holds, the waiting times at node 1 becomes zero after finite number of
services  there, and stays zero thereafter.\end{prop}
\begin{proof}
Recall that~\eqref{eq:det-A-n},~\eqref{eq:det-gamma-i+K},~\eqref{eq:det-tau-equation}, and~\eqref{eq:det-V-n+1} hold. We show
that after maximal $K$ further services at node $1$ the waiting
time there will be $0$.

If one of the waiting times $V_1(n)$, $n\in\{1,\ldots,K+1\}$ is $0$
then by the property~\eqref{eq:det-decreasing-V-n} we are done. We
show by contradiction, that the case  $V_1(n)>0$ for all $n\in\{1,\ldots,K+1\}$
is not possible.
By~\eqref{eq:det-V-n-positive-recursion} we have
\begin{align*}
\sum_{n=1}^{K}V_1(n+1)= &
\sum_{n=1}^{K}V_1(n)+K\mmu_{1}-
\sum_{n=1}^{K}A_1(n) \\
\overset{\eqref{eq:det-sum-An-V-1}}{\Longleftrightarrow}
\sum_{n=1}^{K}V_1(n+1)= &
\sum_{n=2}^{K}V_1(n)+K\mmu_{1}
-(\mmu_{1}+\mmu_{2}+\mmu_{3}+\mmu_{4}) \\
 V_1(K+1) = & K\mmu_{1}-(\mmu_{1}+\mmu_{2}+\mmu_{3}+\mmu_{4}).
\end{align*}

By assumption~\eqref{eq:no-waiting-time-conditions} we have
\[
V_1(K+1)\leq 0
\]
 which contradicts the assumption $V_1(n)>0$ for all $n\in\{1,\ldots,K+1\}$.\end{proof}
\begin{prop}
\label{prop:det-supersaturated}For a cycle with $K$ customers, such that
\begin{equation}
K\mmu_{1}>\mmu_{1}+\mmu_{2}+\mmu_{3}+\mmu_{4}\label{eq:no-waiting-time-conditions-1}
\end{equation}
holds,  after finite number of services at node $1$
some customers waiting time at  node $1$ is positive. Thereafter, all subsequent customers' waiting times at  node $1$ are positive as well.\end{prop}
\begin{proof}
Recall that~\eqref{eq:det-A-n},~\eqref{eq:det-gamma-i+K},~\eqref{eq:det-tau-equation},
and~\eqref{eq:det-V-n+1} hold. From~\eqref{eq:det-A-n} and~\eqref{eq:det-sum-An-V-1}  for
all $n\in\mathbb{N}$  follows
\[
V_1(n)+\mmu_{1}+\mmu_{2}+\mmu_{3}+\mmu_{4}\geq K\mmu_{1}.
\]
By assumption~\eqref{eq:no-waiting-time-conditions-1} we have
\[
V_1(n)+K\mmu_{1}>K\mmu_{1},\qquad\forall n\in\mathbb{N}.
\]
\end{proof}

\begin{prop}
For a cycle with $K$ customers  the asymptotic average waiting
times $\bar{V}_{j}$ at nodes $j\in\{1,2,3,4\}$ are
\begin{align}
\bar{V}_{1} & =\max\left(0,K\mmu_{1}-\sum_{i=1}^{4}\mmu_{i}\right)\label{eq:det-V-bar-1}.\\
\bar{V}_{j} & =0,\qquad j\in\{2,3,4\}.\label{eq:det-V-bar-j-not-1}
\end{align}

The asymptotic idle probability $\pi_{1}^{(K)}(0)$ for node $1$ is
\begin{equation}
\pi_{1}^{(K)}(0)=1-\frac{K\mmu_{j}}{\sum_{i=1}^{4}\mmu_{i}+\bar{V}_{1}}\label{eq:pi-1-0-deterministic-1}.
\end{equation}
\end{prop}
\begin{proof}
From the very definition of infinite servers $\bar{V}_{2} = \bar{V}_{4}=0$,
and we already showed that after a finite initial period waiting times
at node $3$ are zero, thus~\eqref{eq:det-V-bar-j-not-1}
holds.

\begin{ignorechecktex}
\begin{enumerate}[label=(\roman*)]
\item\label{enu:Kmu-less}
In case of $K\mmu_{1}\leq\mmu_{1}+\mmu_{2}+\mmu_{1}+\mmu_{4}$ we proved
in \prettyref{prop:det-not-saturated-and-saturation-point} that $V_{1}(n)=0$
for all $n>K$, so $\bar{V}_{1}=0$.

\item\label{enu:Kmu-greater-than-or-equal-to}
In case of $K\mmu_{1}>\mmu_{1}+\mmu_{2}+\mmu_{1}+\mmu_{4}$
we proved in \prettyref{prop:det-supersaturated} that after a finite initial period
$V_{1}(n)>0$ holds for all $n$.
\end{enumerate}
\end{ignorechecktex}

So node $1$ stays busy after this initial period  and the inter-departure times  from node $1$ are
$\mmu_{1}$. Because the backcycle time is a constant delay  $\mmu_{2}+\mmu_{3}+\mmu_{4}$, the inter-arrival times at node $1$ are $A_1(n)=\mmu_{1}$ as well and when the $n$-th customer returns to the
node 1 the inter-arrival time $A_{n+K}$ is $\mmu_{1}$. Substituting
these results into~\eqref{eq:det-sum-An-V-1} we obtain for all customers $n$ arriving at $1$
after the finite initial period
\begin{align*}
\sum_{i=n}^{n+K-1}A_1(i) =
K\mmu_{1} & =V_{1}(n+1+K)+\mmu_{1}+\mmu_{2}+\mmu_{3}+\mmu_{4}\\
\Longrightarrow\bar{V}_{1} & =K\mmu_{1}-(\mmu_{1}+\mmu_{2}+\mmu_{3}+\mmu_{4}).
\end{align*}
Summarizing the results of~\ref{enu:Kmu-less} and~\ref{enu:Kmu-greater-than-or-equal-to}
yields~\eqref{eq:det-V-bar-1}.

The formula~\eqref{eq:pi-1-0-deterministic-1} is obtained by Boyse and Warn~\cite[(2.1)]{boyse;warn:75}.
\end{proof}
\bibliographystyle{alpha}
\bibliography{}

\end{document}

%% file: table-p0-values-no-d.tex
\begingroup\small
\begin{tabular}{rrrrrrrrrrr}
  \toprule
 & K=1 & K=2 & K=3 & K=4 & K=5 & K=6 & K=7 & K=8 & K=9 & K=10 \\ 
  \midrule
simulation & 0.880 & 0.762 & 0.646 & 0.533 & 0.424 & 0.319 & 0.219 & 0.129 & 0.056 & 0.014 \\ 
  FLOW & 0.880 & 0.760 & 0.640 & 0.520 & 0.400 & 0.280 & 0.160 & 0.040 & 0.000 & 0.000 \\ 
  MVA & 0.880 & 0.765 & 0.656 & 0.553 & 0.458 & 0.372 & 0.296 & 0.229 & 0.174 & 0.129 \\ 
  ST\&ST & 0.880 & 0.760 & 0.641 & 0.525 & 0.416 & 0.315 & 0.223 & 0.142 & 0.075 & 0.029 \\ 
  GMVA & 0.867 & 0.738 & 0.613 & 0.494 & 0.382 & 0.278 & 0.185 & 0.105 & 0.040 & 0.000 \\ 
  ESUM & 0.880 & 0.766 & 0.654 & 0.548 & 0.449 & 0.359 & 0.281 & 0.215 & 0.163 & 0.124 \\ 
  EBOTT & 0.880 & 0.766 & 0.654 & 0.548 & 0.449 & 0.359 & 0.281 & 0.215 & 0.163 & 0.124 \\ 
   \bottomrule
\end{tabular}
\endgroup

%% file: table-p0-abs-error-no-d.tex
\begingroup\small
\begin{tabular}{rllllllllll}
  \toprule
 & K=1 & K=2 & K=3 & K=4 & K=5 & K=6 & K=7 & K=8 & K=9 & K=10 \\ 
  \midrule
FLOW & \textbf{0.000} & \textbf{0.002} & 0.006 & 0.013 & 0.024 & 0.039 & 0.059 & 0.089 & 0.056 & \textbf{0.014} \\ 
  MVA & \textbf{0.000} & 0.003 & 0.009 & 0.020 & 0.034 & 0.053 & 0.076 & 0.100 & 0.118 & 0.114 \\ 
  ST\&ST & \textbf{0.000} & \textbf{0.002} & \textbf{0.005} & \textbf{0.008} & \textbf{0.007} & \textbf{0.004} & \textbf{0.004} & \textbf{0.013} & 0.019 & \textbf{0.014} \\ 
  GMVA & 0.013 & 0.024 & 0.033 & 0.039 & 0.042 & 0.041 & 0.034 & 0.024 & \textbf{0.016} & \textbf{0.014} \\ 
  ESUM & \textbf{0.000} & 0.003 & 0.008 & 0.015 & 0.025 & 0.040 & 0.061 & 0.086 & 0.107 & 0.110 \\ 
  EBOTT & \textbf{0.000} & 0.003 & 0.008 & 0.015 & 0.025 & 0.040 & 0.061 & 0.086 & 0.107 & 0.110 \\ 
   \bottomrule
\end{tabular}
\endgroup

%% file: table-p0-values-with-d-1d30.tex
\begingroup\small
\begin{tabular}{rrrrrrrrrrr}
  \toprule
 & K=1 & K=2 & K=3 & K=4 & K=5 & K=6 & K=7 & K=8 & K=9 & K=10 \\ 
  \midrule
simulation & 0.869 & 0.743 & 0.623 & 0.508 & 0.399 & 0.296 & 0.202 & 0.117 & 0.050 & 0.013 \\ 
  FLOW & 0.870 & 0.739 & 0.609 & 0.478 & 0.348 & 0.218 & 0.087 & 0.000 & 0.000 & 0.000 \\ 
  MVA & 0.870 & 0.745 & 0.628 & 0.519 & 0.419 & 0.330 & 0.253 & 0.188 & 0.136 & 0.095 \\ 
  ST\&ST & 0.870 & 0.742 & 0.638 & 0.557 & 0.493 & 0.439 & 0.393 & 0.353 & 0.318 & 0.287 \\ 
  GMVA & 0.883 & 0.773 & 0.672 & 0.579 & 0.497 & 0.424 & 0.361 & 0.308 & 0.263 & 0.227 \\ 
  ESUM & 0.870 & 0.754 & 0.647 & 0.550 & 0.465 & 0.391 & 0.328 & 0.276 & 0.232 & 0.196 \\ 
  EBOTT & 0.870 & 0.754 & 0.647 & 0.550 & 0.465 & 0.391 & 0.328 & 0.276 & 0.232 & 0.196 \\ 
  ST\&ST-m & 0.870 & 0.742 & 0.619 & 0.502 & 0.394 & 0.295 & 0.207 & 0.131 & 0.069 & 0.026 \\ 
   \bottomrule
\end{tabular}
\endgroup

%% file: table-p0-abs-error-with-d-1d30.tex
\begingroup\small
\begin{tabular}{rllllllllll}
  \toprule
 & K=1 & K=2 & K=3 & K=4 & K=5 & K=6 & K=7 & K=8 & K=9 & K=10 \\ 
  \midrule
FLOW & \textbf{0.000} & 0.004 & 0.014 & 0.030 & 0.051 & 0.078 & 0.114 & 0.117 & 0.050 & \textbf{0.013} \\ 
  MVA & \textbf{0.000} & 0.002 & 0.005 & 0.010 & 0.020 & 0.034 & 0.051 & 0.071 & 0.086 & 0.082 \\ 
  ST\&ST & \textbf{0.000} & \textbf{0.001} & 0.015 & 0.049 & 0.094 & 0.143 & 0.192 & 0.236 & 0.269 & 0.275 \\ 
  GMVA & 0.014 & 0.030 & 0.049 & 0.071 & 0.098 & 0.128 & 0.159 & 0.191 & 0.214 & 0.214 \\ 
  ESUM & \textbf{0.000} & 0.011 & 0.024 & 0.042 & 0.066 & 0.095 & 0.127 & 0.158 & 0.182 & 0.183 \\ 
  EBOTT & \textbf{0.000} & 0.011 & 0.024 & 0.042 & 0.066 & 0.095 & 0.127 & 0.158 & 0.182 & 0.183 \\ 
  ST\&ST-m & \textbf{0.000} & \textbf{0.001} & \textbf{0.004} & \textbf{0.006} & \textbf{0.005} & \textbf{0.001} & \textbf{0.005} & \textbf{0.014} & \textbf{0.019} & 0.014 \\ 
   \bottomrule
\end{tabular}
\endgroup

%% file: bergbau.bbl
\begin{thebibliography}{BCMP75}

\bibitem[Bar76]{barbour:76}
A.~D. Barbour.
\newblock Networks of queues and the method of stages.
\newblock {\em Advances in Applied Probability}, 8:584--591, 1976.

\bibitem[BCMP75]{baskett;chandy;muntz;palacios:75}
F.~Baskett, M.~Chandy, R.~Muntz, and F.G. Palacios.
\newblock Open, closed and mixed networks of queues with different classes of
  customers.
\newblock {\em Journal of the Association for Computing Machinery},
  22:248--260, 1975.

\bibitem[BGMT06]{bolch;greiner;demeer;trivedi:06}
G.~Bolch, S.~Greiner, H.~{de Meer}, and K.~S. Trivedi.
\newblock {\em Queueing networks and {M}arkov chains}.
\newblock John Wiley, New York, 2 edition, 2006.

\bibitem[BW75]{boyse;warn:75}
J.~W. Boyse and D.~R. Warn.
\newblock A straightforward model for computer performance prediction.
\newblock {\em ACM Comput. Surv.}, 7(2):73--93, June 1975.

\bibitem[Car86]{carmichael:86}
D.G. Carmichael.
\newblock Shovel–truck queues: a reconciliation of theory and practice.
\newblock {\em Construction Management and Economics}, 42:161--177, 1986.

\bibitem[Car87]{carmichael:87}
D.G. Carmichael.
\newblock {\em Engineering queues in construction and mining}.
\newblock Ellis Horwood series in civil engineering. Ellis Horwood, Chichester,
  1987.

\bibitem[CH11]{chanda;hardy:11}
E.K. Chanda and R.J. Hardy.
\newblock Selection criteria for open pit production equipment – payload
  distributions and the {“10/10/20”} policy.
\newblock In {\em 35th {APCOM} Symposium 2011}, pages 24--30, Carlton,
  Victoria, 2011. The Australasian Institute of Mining and Metallurgy.

\bibitem[Cza08]{czaplicki:08}
J.M. Czaplicki.
\newblock {\em Shovel-Truck Systems: Modelling, Analysis and Calculations}.
\newblock A Balkema Book. CRC Press, Taylor \& Francis, Boca Raton, 2008.

\bibitem[Dad86]{daduna:86a}
H.~Daduna.
\newblock Two-stage cyclic queues with nonexponential servers: Steady state and
  cycle time.
\newblock {\em Operations Research}, 32(3):455--459, 1986.

\bibitem[Dad01]{daduna:01a}
H.~Daduna.
\newblock Stochastic networks with product form equilibrium.
\newblock In D.N. Shanbhag and C.R. Rao, editors, {\em Stochastic Processes:
  Theory and Methods}, volume~19 of {\em Handbook of Statistics}, chapter~11,
  pages 309--364. Elsevier Science, Amsterdam, 2001.

\bibitem[Gau12]{gautam:12}
N.~Gautam.
\newblock {\em Analysis of Queues: Methods and Applications}.
\newblock Operations Research Series. CRC Press, Boca Raton, 2012.

\bibitem[GH74]{gross;harris:74}
D.~Gross and C.M. Harris.
\newblock {\em Fundamentals of {Q}ueueing {T}heory. {S}econd {E}dition}.
\newblock John Wiley and Sons, Inc., Chichester -- New York -- Brisbane --
  Toronto -- Singapore, 1974.

\bibitem[GN67]{gordon;newell:67}
W.J. Gordon and G.F. Newell.
\newblock Closed queueing networks with exponential servers.
\newblock {\em Operations Research}, 15:254--265, 1967.

\bibitem[Jac57]{jackson:57}
J.R. Jackson.
\newblock Networks of waiting lines.
\newblock {\em Operations Research}, 5:518--521, 1957.

\bibitem[jmt91]{jmt:v0.9.1}
{J}{M}{T} --- {J}ava {M}odelling {T}ools.
\newblock \url{http://jmt.sourceforge.net}, Version 0.9.1.

\bibitem[Kel76]{kelly:76}
F.~Kelly.
\newblock Networks of queues.
\newblock {\em Advances in Applied Probability}, 8:416--432, 1976.

\bibitem[Kel79]{kelly:79}
F.~P. Kelly.
\newblock {\em Reversibility and Stochastic Networks}.
\newblock John Wiley and Sons, Chichester -- New York -- Brisbane -- Toronto,
  1979.

\bibitem[KP10]{knights;paton:10}
P.~Knights and S.~Paton.
\newblock Payload variance effects on truck bunching.
\newblock In {\em Seventh Large Open Pit Mining Conference}, pages 27--28,
  Perth, 2010.

\bibitem[KY91]{kappas;yegulalp:91}
G.~Kappas and T.~M. Yegulalp.
\newblock An application of closed queueing networks theory in truck-shovel
  systems.
\newblock {\em International Journal of Surface Mining, Reclamation and
  Environment}, 5:45--53, 1991.

\bibitem[LR80]{lavenberg;reiser:80}
S.~S. Lavenberg and M.~Reiser.
\newblock Stationary state probabilities at arrival instants for closed
  queueing networks with multiple types of customers.
\newblock {\em Journal of Applied Probability}, 17:1048--1061, 1980.

\bibitem[Mud97]{muduli:97}
P.K. Muduli.
\newblock {\em Modeling truck-shovel systems as multiple-chain closed queuing
  networks}.
\newblock PhD thesis, Columbia University, New York, 1997.

\bibitem[Pan93]{panagiotou:93}
G.N. Panagiotou.
\newblock Optimizing the shovel-truck operation using simulation and queueing
  models.
\newblock In {\em Proceedings of the Second International Symposium on Mine
  Mechanization and Automation}, Lulea, Sweden, 1993.

\bibitem[PZX88]{peng;zhang;xi:88}
S.~Peng, D.~Zhang, and Y.~Xi.
\newblock Computer simulation of a semi-continuous open-pit mine haulage
  system.
\newblock {\em Int. J. Min. Geol. Eng.}, 6:267--271, 1988.
\newblock doi:10.1007/BF00880978.

\bibitem[RL80]{reiser;lavenberg:80}
M.~Reiser and S.S. Lavenberg.
\newblock Mean value analysis of closed multichain queueing networks.
\newblock {\em Journal of the Association for Computing Machinery},
  27:313--322, 1980.

\bibitem[Sch73]{schassberger:73}
R.~Schassberger.
\newblock {\em Warteschlangen}.
\newblock Springer, Wien, 1973.

\bibitem[sim05]{simpy:v3.0.5}
{S}im{P}y.
\newblock \url{https://pypi.python.org/pypi/simpy}, Version 3.0.5.

\bibitem[SM81]{sevcik;mitrani:81}
K.~C. Sevcik and I.~Mitrani.
\newblock The distribution of queueing network states at input and output
  instants.
\newblock {\em Journal of the Association for Computing Machinery},
  28:358--371, 1981.

\bibitem[SS71]{stoyan;stoyan:71}
D.~Stoyan and H.~Stoyan.
\newblock {\em Mathematische Methoden in der Operationsforschung:
  F{\"o}rdertechnik, Bergbau, Transportwesen}.
\newblock {VEB} Deutscher Verlag f{\"u}r Grundstoffindustrie, Leipzig, 1971.

\bibitem[{Sti}72]{stidham:72}
S.~{Stidham Jr.}
\newblock {$L = \lambda \cdot W$}: A discounted analogue and a new proof.
\newblock {\em Operations Research}, 20:1115--1126, 1972.

\bibitem[{Sti}74]{stidham:74}
S.~{Stidham Jr.}
\newblock A last word on {$L = \lambda \cdot W$}.
\newblock {\em Operations Research}, 22:417--421, 1974.

\bibitem[Sto78]{stoyan:78}
D.~Stoyan.
\newblock Queueing networks - insensitivity and a heuristic approximation.
\newblock {\em Elektronische {I}nformationsverarbeitung und {K}ybernetik},
  14(3):135--143, 1978.

\bibitem[TIJ13]{ta;ingolfsson;doucette:13}
C.~H. Ta, A.~Ingolfsson, and Doucette J.
\newblock A linear model for surface mining haul truck allocation incorporating
  shovel idle probabilities.
\newblock {\em European Journal of Operational Research}, 231(3):770--778,
  2013.

\bibitem[Whi84]{whitt:84a}
W.~Whitt.
\newblock Open and closed moodels for networks of queues.
\newblock {\em {AT\& T} Bell Laboratories Technical Journal}, 63(9):1911--1979,
  1984.

\bibitem[ZW09]{zhang;wang:09}
J.~Zhang and Q.~Wang.
\newblock A queuing network model for shovel-truck-crusher systems in open-pit
  mining.
\newblock In {\em Proceedings of Application of Computers and Operations
  Research in the Mineral Industry Conference}, pages 18--26, Vancouver,
  Canada, 2009. Canadian Institute of Mining, Metallurgy, and Petroleum.

\end{thebibliography}
